\documentclass[11pt]{article}

\usepackage[colorlinks=true,bookmarks=true,linkcolor=blue,urlcolor=blue,citecolor=blue,breaklinks=true]{hyperref}
\usepackage{breakcites}

\usepackage{doi}

\usepackage{amsmath}
\usepackage{amsfonts}
\usepackage{amsthm}
\usepackage{amssymb}

\usepackage{mathtools}

\usepackage{xcolor}
\usepackage{bm} 
\usepackage{dsfont} 

\usepackage{graphicx}
\usepackage{tikz-cd}
\usepackage{tikz}
\definecolor{mycolor1}{rgb}{0.105882,0.619608,0.466667}
\definecolor{mycolor2}{rgb}{0.85098,0.372549,0.00784314}
\definecolor{mycolor3}{rgb}{0.458824,0.439216,0.701961}
\definecolor{mycolor4}{rgb}{0.905882,0.160784,0.541176}
\definecolor{mycolor5}{rgb}{0.4,0.65098,0.117647}
\definecolor{mycolor6}{rgb}{0.65098,0.462745,0.113725}
\definecolor{mycolor7}{rgb}{0.901961,0.670588,0.00784314}
\definecolor{mycolor8}{rgb}{0.4,0.4,0.4}
\definecolor{mycolor9}{rgb}{0.301961,0,0.294118}
\definecolor{mycolor10}{rgb}{0.0313725,0.25098,0.505882}

\usepackage{psfrag}

\usepackage{algorithm}
\usepackage{algpseudocode}
\algdef{SE}[DOWHILE]{Do}{doWhile}{\algorithmicdo}[1]{\algorithmicwhile\ #1}%

\usepackage{mathrsfs} 

\DeclareMathOperator*{\argmin}{arg\,min}

\newcommand{\spann}{\mathrm{span}}

\newcommand{\Exp}{\mathrm{Exp}}
\newcommand{\Log}{\mathrm{Log}}

\newcommand{\Retr}{\mathrm{R}}
\newcommand{\inj}{\mathrm{inj}}

\newcommand{\T}{\mathrm{T}}

\newcommand{\Rd}{{\mathbb{R}^{d}}}

\newcommand{\reals}{{\mathbb{R}}}

\newcommand{\Rn}{{\mathbb{R}^n}}

\newcommand{\grad}{\mathrm{grad}}
\newcommand{\Hess}{\mathrm{Hess}}

\newcommand{\D}{\mathrm{D}}

\newcommand{\calF}{\mathcal{F}}

\newcommand{\calM}{\mathcal{M}}

\newcommand{\rank}{\operatorname{rank}}

\newcommand{\TODOF}[1]{}

\newcommand{\dist}{\mathrm{dist}}

\newcommand{\graph}{\mathrm{graph}}
\newcommand{\Eu}{E_{\mathrm{u}}}
\newcommand{\Ecs}{E_{\mathrm{cs}}}
\newcommand{\Iu}{I_{\mathrm{u}}}
\newcommand{\Ics}{I_{\mathrm{cs}}}
\newcommand{\Wcsloc}{W_{\mathrm{cs}}^\mathrm{loc}}
\newcommand{\Wcsloctilde}{\widetilde{W}_{\mathrm{cs}}^\mathrm{loc}}

\newcommand{\Lip}{\mathrm{Lip}}
\newcommand{\TEu}{T\vert_{\Eu}}
\newcommand{\pcs}{p_{\mathrm{cs}}}
\newcommand{\pu}{p_{\mathrm{u}}}
\newcommand{\gcs}{g_{\mathrm{cs}}}
\newcommand{\gu}{g_{\mathrm{u}}}
\newcommand{\Tcs}{T_{\mathrm{cs}}}
\newcommand{\Tu}{T_{\mathrm{u}}}

\newcommand{\hyvarphi}{h_y^\varphi}
\newcommand{\Fix}{\mathrm{Fix}}

\newcommand{\positive}{(0, \infty)} 

\newtheorem{theorem}{Theorem}[section]
\newtheorem{lemma}[theorem]{Lemma}
\newtheorem{proposition}[theorem]{Proposition}

\newtheorem{definition}[theorem]{Definition} 
\newtheorem{remark}[theorem]{Remark}

\usepackage[lmargin=3cm,rmargin=3cm,bottom=3cm,top=3cm]{geometry}
\usepackage[ansinew]{inputenc}
\usepackage[round,compress]{natbib}

\usepackage{sectsty}
\makeatletter\def\@seccntformat#1{\protect\makebox[0pt][r]{\csname the#1\endcsname\hspace{12pt}}}\makeatother

\usepackage{subcaption}

\usepackage{enumitem} 

\usepackage[normalem]{ulem} 

\usepackage{multicol}
\usepackage{multirow}

\usepackage{blindtext} 

\newcommand{\aref}[1]{\hyperref[#1]{A\ref{#1}}}
\newtheorem{assumption}{A\ignorespaces} 

\usepackage{xcolor}
\usepackage{tikz}
\usepackage{comment}

\usepackage{etoolbox}
\newbool{BanachAppendix}
\setbool{BanachAppendix}{false} 

\title{A non-autonomous center-stable set theorem\\for saddle avoidance in optimization}

\author{
  Andreea-Alexandra Mu\c{s}at and Nicolas Boumal\thanks{Ecole Polytechnique F\'ed\'erale de Lausanne (EPFL), Institute of Mathematics,
      \texttt{\{andreea.musat, nicolas.boumal\}@epfl.ch}.}
}
\date{\today}

\begin{document}

\maketitle

\begin{abstract}
Optimization algorithms are unlikely to converge to strict saddle points.
Proofs to that effect rely on the Center-Stable Manifold Theorem (CSMT), casting algorithms as dynamical systems: $x_{k+1} = g_k(x_k)$.
In its standard form, the CSMT is limited to autonomous systems (the maps $g_k$ are all the same).
To study algorithms such as gradient descent with non-constant step-size schedules, we need a non-autonomous CSMT.
There are a few, but they are unable to handle, for example, vanishing step sizes.
To cover such scenarios, we establish a new Center-Stable Set Theorem (CSST) for non-autonomous systems.
We use it to prove saddle avoidance for gradient descent (Euclidean and Riemannian) and for the proximal point method, without assuming Lipschitz gradients or isolated saddles, and allowing vanishing step sizes.
\end{abstract}


\section{Introduction}

Empirically, unconstrained optimization algorithms tend to avoid the strict saddle points of the objective function $f \colon \reals^d \to \reals$, that is, they generically do not converge to critical points where the Hessian has at least one negative eigenvalue.

A standard way to study this behavior is to view the algorithm as a discrete-time dynamical system: a map $g \colon \reals^d \to \reals^d$ that produces the next iterate as $x_{k+1} = g(x_k)$.
This perspective is useful because it enables the use of tools from dynamical systems theory, such as the \emph{center-stable manifold theorem} (CSMT) \citep[Thm.~5.1]{hirsch1977invariant}, \citep[Thm.~III.7]{shub1987book}.

This encompasses algorithms such as gradient descent (GD) with a constant step size $\alpha > 0$, which indeed uses a \emph{single update map} $g(x) = x - \alpha \nabla f(x)$, applied at all times $k$.
This \emph{autonomous} setting is well understood and forms the basis of most saddle avoidance results.

Some algorithms, however, are \emph{non-autonomous}: the update rule changes with time.
Explicitly, there is a predetermined sequence of maps $g_0, g_1, \dots$ and iterates satisfy $x_{k+1} = g_k(x_k)$.
For example, GD with vanishing step sizes fits into this framework: the iteration maps are $g_k(x) = x - \alpha_k \nabla f(x)$, where the step sizes $\alpha_k > 0$ are fixed ahead of time (no line-search) and converge to zero.

\cite{panageas2019first} argue that GD with vanishing step sizes $\alpha_k = \Omega(\frac{1}{k})$ avoids strict saddles for generic initial points, under the assumptions that all saddles are isolated
and all step sizes are small enough so that each map $g_k$ is a diffeomorphism (see also \emph{Related work}).
We set out to allow for non-isolated saddles and arbitrarily large step sizes (removing any Lipschitz gradient assumptions).
The plan is to do so by developing a general theorem for non-autonomous dynamical systems first, and applying it to optimization second.

Since the CSMT does not apply in the setting of GD with vanishing steps, \cite{panageas2019first} proposed a tailored theorem to claim that, locally, the set of initial points converging to an isolated strict saddle has measure zero.
In this setting, a tailored version of the CSMT is indeed required.
While there exist non-autonomous extensions of the CSMT that handle sequences of iteration maps~\citep[Thm.~6.2.8]{katok1995introduction}, these results impose several technical assumptions on the maps.
In particular, they require that all linearizations of the iteration maps contract and expand uniformly, with the same constants, on fixed subspaces.
GD with vanishing step sizes does not satisfy these conditions: as the step size converges to zero, the iteration maps approach identity, as do their differentials: no uniform expansion or contraction constants can exist in this setting.

Our main contributions are the following:
\begin{enumerate}[leftmargin=0pt]
    \item We first develop a non-autonomous center-stable \emph{set}\footnote{The autonomous CSMT also establishes that the center-stable \emph{set} is a \emph{manifold}.
    In our case, it is unclear if the set has higher regularity; however, for our purposes, it only matters that it has measure zero, as shown.}
    theorem (Theorem~\ref{theorem:csmt-general}) for sequences of iteration maps whose linearizations contract and expand along fixed subspaces, but not necessarily with uniform constants across the sequence.
    The theorem shows that, locally, the set of initial conditions that yield convergence to special unstable fixed points (Definition~\ref{def:unstable-fixed-point}) has measure zero.
    We extend it to systems on manifolds in Section~\ref{subsec:riemannian-reduction}.
    \item We then show that for optimization algorithms such as GD (both on Euclidean spaces and Riemannian manifolds) and the proximal point method (PP) with vanishing step-size schedules (among others), the strict saddles of $f$ correspond indeed to such unstable fixed points, so the theorem (together with a globalization argument) implies that these algorithms avoid strict saddle points (Theorems~\ref{thm:gd-vanishing-theorem}, \ref{thm:RGDmain} and~\ref{thm:pp-avoidance}).
    In particular, our results do not require the saddles to be isolated, and (for GD) the initial step sizes can be arbitrarily large.
\end{enumerate}

\subsection*{Proof technique}

Our proof builds on that of the autonomous CSMT by \citet{hirsch1977invariant}, which uses the graph transform---a method that goes back to Hadamard and Perron (see \emph{Related work} below).
In this standard setting where $x_{k+1} = g(x_k)$, for each fixed point $x^\ast$, the argument produces a single function $\varphi$ whose graph (a subset of $\reals^d$) contains $x^\ast$ and is invariant under $g$, that is, $g(\graph(\varphi)) \subseteq \graph(\varphi)$.
It is designed so that, if a trajectory $(x_k)_{k \geq 0}$ converges to $x^\ast$, then there exists some $\bar{k}$ such that $x_k \in \graph(\varphi)$ for all $k \ge \bar{k}$. This shows that $x^\ast$ can only be approached along this single smooth set.
The function $\varphi$ is constructed as the fixed point of a map $\Gamma$ on a space of functions (the ``graph transform'').
The graph of $\varphi$ is the so-called center-stable manifold of $x^\ast$.
If $x^\ast$ is an unstable fixed point, the graph has positive codimension and hence measure zero.

The non-autonomous case differs in a fundamental way: the iteration map changes with time, so we cannot expect a single fixed function $\varphi$ to describe the invariant set.
Instead, we construct a sequence of Lipschitz functions $\varphi_0, \varphi_1, \dots$ such that $g_k(\graph(\varphi_k)) \subseteq \graph(\varphi_{k+1})$.
In analogy to the autonomous case, we design them so that, if a trajectory $(x_k)_{k \geq 0}$ converges to $x^\ast$, then there exists some $\bar{k}$ such that $x_k \in \graph(\varphi_k)$ for all $k \geq \bar{k}$.
Unlike before, the approach to $x^\ast$ is not constrained to any single graph; rather, it consists of jumps through a sequence of graphs.
The functions $\varphi_k$ are Lipschitz, hence their graphs have measure zero: this is key to the avoidance claim.

We present the essential building blocks of the proof of the autonomous CSMT in Section~\ref{section:graph-transform-autonomous}, and we repurpose them for the non-autonomous case in Section~\ref{section:graph-transform-nonautonomous}.
The argument relies on a potential function $V_{\varphi_k}$ that measures the distance to $\graph(\varphi_k)$.
This potential explodes along trajectories that are not confined to those graphs.
Therefore, trajectories that remain near an unstable fixed point (and hence whose potential is bounded) must lie on the graphs.

Our non-autonomous center-stable set theorem provides a \emph{local} saddle avoidance result. 
We globalize it with an argument based on the Luzin $N^{-1}$ property (Definition~\ref{def:luzin}).

This all results in Theorem~\ref{thm:global-avoidance-nonautonomous}, which is a general statement about dynamical systems.
To obtain statements about optimization algorithms, it then suffices to verify that the resulting iteration maps meet the hypotheses of that theorem: see Sections~\ref{section:gd}, \ref{sec:extension-manifolds} and~\ref{section:pp}.

\subsection*{Related work}

\paragraph{The standard CSMT and some variants.}  
The study of invariant manifolds has its roots in classical work on differential equations \citep{poincare1881memoire}.
Subsequent rigorous analyses near equilibria were developed by \citet{hadamard1901iteration} (see \citep{hasselblatt2017iteration} for a translation) and \citet{perron1929stabilitat}.
A comprehensive treatment for discrete-time systems was later given by \citet{hirsch1977invariant}.
These underpin the modern analyses of local dynamics near strict saddle points in optimization.
For a more detailed historical overview, see the endnotes of \citep[Ch.~6]{kloeden2011nonautonomous} and \citep[Ch.~1]{hirsch1977invariant}.
We are most interested in non-autonomous versions of the CSMT: we found a few, including work by \citep[Thm.~6.2.8]{katok1995introduction} and \citep[\S3.4]{dyatlov2018notes}, but they are not applicable in our case.

\paragraph{Applications in optimization.}  
Several works study how optimization algorithms avoid strict saddle points. 
\citet{goudou2009gradient} showed that a continuous-time heavy-ball flow generically converges to local minima of $C^2$ coercive Morse functions.
A more recent line of work showed that gradient descent with suitable step sizes avoids strict saddles \citep{lee2016gradient,panageas2017gradient, schaeffer2020extending}.
This was later extended to other methods, including coordinate descent, block coordinate descent, mirror descent, Riemannian gradient descent \citep{lee2019strictsaddles}, accelerated methods \citep{o2019behavior}, alternating minimization \citep{li2019alternating} and non-smooth settings \citep{davis2019active, cheridito2024gradient}.

Certain non-autonomous settings have also been studied.
Closest to us, \citet{panageas2019first} argue that GD and PP with vanishing step sizes avoid \emph{isolated} strict saddles if \emph{all} step sizes are small enough.\footnote{Earlier versions of their paper claim the non-isolated case too. In communication with the authors, they confirmed that their proof does not encompass that setting.} 
\citet{schaeffer2020extending} also analyzed GD with non-constant step sizes, assuming either a contractive $C^1$ update rule $\alpha_{k+1} = h(\alpha_k)$ with \emph{a unique fixed point} $\alpha^\ast > 0$, or a piecewise-constant sequence with finitely many jumps.
However, their result too requires that $\nabla f$ is globally $L$-Lipschitz and that $\alpha_k L \in (0, 2)$ for all $k$.
In particular, the contractive-update setting does not cover vanishing step sizes, which correspond to a fixed point $\alpha^\ast = 0$.

For an example of GD where the step sizes are adaptive (not pre-set), see our analysis of a modified line-search method in \citep{musat2025saddlelinesearch}.
There, different initial points $x_0$ may result in different step-size sequences, which calls for different techniques.

\subsection*{Notation}

Let $(X,d_X)$ and $(Y,d_Y)$ be metric spaces.
For a map $g \colon X \to Y$ and a subset $U \subseteq X$, we write $g\vert_U$ for the restriction of $g$ to $U$ and define the Lipschitz constant of $g\vert_U$ (in $[0, +\infty]$) as 
\begin{align*}
    \Lip(g\vert_U) \coloneqq
    \sup_{\substack{x,y \in U\\ x \neq y}} \frac{ d_Y(g(x),g(y)) }{ d_X(x,y) }.
\end{align*}
Following standard notation, $\reals^d = V \oplus W$ means $V, W$ are two linear subspaces of $\reals^d$ and each vector $ x \in \reals^d $ can be written uniquely as $x = v + w$ with $v \in V$ and $w \in W$.
Given $\varphi \colon V \to W$, we define the \emph{graph of $\varphi$} as 
\begin{align}
    \graph(\varphi) \coloneqq
    \{ v + \varphi(v) \mid v \in V \} =
    \{ x \in V \oplus W \mid x = v + w \textrm{ with } w = \varphi(v) \}.
    \label{eq:graphnotation}
\end{align}
At times, we also write $x = (v, w)$ to mean $x = v + w$ with $v \in V$ and $w \in W$, so that $x$ is in $\graph(\varphi)$ exactly if $x = (v, \varphi(v))$.
In a metric space (usually $\Rd$ with some norm provided by context), $B_r(x)$ denotes the open ball of radius $r$ around $x$.

\section{A non-autonomous avoidance theorem for unstable fixed point}

In this section, we establish conditions under which a non-autonomous dynamical system avoids certain unstable fixed points.
We first present the main result: a non-autonomous center-stable \emph{set} theorem, which provides a local avoidance guarantee.
Then, we combine it with the Luzin $N^{-1}$ property (Definition~\ref{def:luzin}) to extend the conclusion globally.
In later sections, we prove these theorems and apply them to optimization algorithms.

\subsection{A local result}

In this subsection, we establish the local component of our avoidance theorem.
We build toward the notion of \emph{non-uniform pseudo-hyperbolicity} and present the main result, showing that trajectories that remain near an unstable fixed point must lie on a family of (measure zero) invariant graphs.
Two later sections are dedicated to the proof.
We start with some notions used throughout.

\begin{definition}
    A (discrete-time) \emph{non-autonomous dynamical system} on $\reals^d$ is a sequence of maps $\big( g_k \colon \reals^d \to \reals^d \big)_{k \geq 0}$.
    For a given initial condition $x_0 \in \reals^d$, the corresponding \emph{trajectory} is the sequence $(x_k)_{k \geq 0}$ defined by $x_{k+1} = g_k(x_k)$, $k \geq 0$.
\end{definition}

\begin{definition}
    A point $x^\ast \in \reals^d$ is a \emph{fixed point} of a non-autonomous dynamical system if it is fixed by all maps, that is, $g_k(x^\ast) = x^\ast$ for all $k \geq 0$.
\end{definition}

A non-autonomous dynamical system is said to \emph{avoid} a set $S$ if the set of initial points $x_0$ whose trajectories converge to a point in $S$ has (Lebesgue) measure zero.
Precisely:

\begin{definition} \label{def:stable-set}
    The \emph{stable set} of a fixed point $x^\ast$ of a non-autonomous dynamical system $\big( g_k \colon \reals^d \to \reals^d \big)_{k \geq 0}$ is the set of initial points whose trajectories converge to $x^\ast$, that is
    \begin{align*}
        W(x^\ast) = \{\, x_0 \in \reals^d \mid \lim_{k \to \infty} x_k = x^\ast \textrm{ where } x_{k+1} = g_k(x_k) \textrm{ for } k \geq 0 \,\}.
    \end{align*}
    We say $( g_k )_{k \geq 0}$ \emph{avoids} $x^\ast$ if $W(x^\ast)$ has measure zero.
    Moreover, $( g_k )_{k \geq 0}$ \emph{avoids a subset} $S \subseteq \reals^d$ if
    $W(S) \coloneqq \bigcup_{x^\ast \in S} W(x^\ast)$ has measure zero. 
\end{definition}

Below, we introduce the standard notion of pseudo-hyperbolicity for a pair $(g, T)$ consisting of a map $g \colon \reals^d \to \reals^d$ and a linear map $T \colon \reals^d \to \reals^d$.  
When $g$ is differentiable with fixed point $g(0) = 0$, it is natural to take $T = \D g(0)$. 

The definition parallels the assumptions of the classical CSMT: the splitting $\reals^d = \Ecs \oplus \Eu$ represents the decomposition into invariant subspaces of $T$, while the pseudo-hyperbolicity conditions ensure that $T$ is non-expanding on vectors in $\Ecs$ (``center-stable'') and expanding on those in $\Eu$ (``unstable''), with respect to some chosen norms on $\Ecs$ and $\Eu$. 

The small-Lipschitz condition requires that $g$ is well approximated by $T$ in some neighborhood of the fixed point.
This is with respect to the max-norm on $\Rd$, induced by the norms on $\Ecs$ and $\Eu$:
\begin{align}
    \|x\|_{\reals^d} \;\coloneqq\; \max\bigl( \|y\|_{\Ecs}, \; \|z\|_{\Eu} \bigr) && \textrm{ where } x = y + z \textrm{ with } y \in \Ecs \textrm{ and } z \in \Eu.
    \label{eq:Rd-norm}
\end{align}
Subscripts on the norms are omitted without risk of confusion.

\begin{definition} \label{def:pseudo-hyperbolicity}
    Let $g \colon \reals^d \to \reals^d$ be a map and $T \colon \reals^d \to \reals^d$ a linear map.  
    Let $\mu, \lambda, \varepsilon \in \reals$ satisfy $1 \leq \lambda < \mu$ and $\varepsilon > 0$.  
    We say that the pair $(g, T)$ is $(\mu, \lambda, \varepsilon)$-\emph{pseudo-hyperbolic} on an open neighborhood $U$ of $0$ with respect to the splitting $\reals^d = \Ecs \oplus \Eu$ and specified norms on $\Ecs$ and $\Eu$, which we write as
    \begin{align*}
        (g, T) \in \mathrm{PH}(\mu, \lambda, \varepsilon; U, \Ecs \oplus \Eu),
    \end{align*}
    if the following conditions hold: 
    \begin{assumption}[Fixed origin] \label{assumption:fixed-origin}
        g(0) = 0.
    \end{assumption}

    \begin{assumption}[Invariance] \label{assumption:invariance}
        If $y \in \Ecs$, then $Ty \in \Ecs$ and similarly if $z \in \Eu$, then $Tz \in \Eu$.
    \end{assumption}

    \begin{assumption}[Pseudo-hyperbolicity] \label{assumption:pseudo-hyperbolic}
        For all $y \in \Ecs$ and all $z \in \Eu$, 
        \begin{align*}
            \| Ty \| \leq \lambda \| y \|  
            \quad\quad \textrm{ and } \quad\quad
            \| T z \| \geq \mu \| z \|.
        \end{align*}
    \end{assumption}

    \begin{assumption}[Small-Lipschitz] \label{assumption:small-lipschitz}
        $\Lip \big( (g-T)\vert_{U} \big) \leq \varepsilon <  (\mu - \lambda) / 4 $ (w.r.t.\ the max-norm~\eqref{eq:Rd-norm}).
    \end{assumption}

    \noindent If $U = \reals^d$, we say that $(g, T)$ is \emph{globally} pseudo-hyperbolic.
\end{definition}

We next extend this notion to sequences of maps, allowing for time-dependent dynamics.

\begin{definition} \label{def:jointlyPHonU}
    A sequence of pairs $\big( (g_k \colon \reals^d \to \reals^d, T_k \colon \reals^d \to \reals^d) \big)_{k \ge 0}$ where each $T_k$ is a linear map is said to be \emph{jointly pseudo-hyperbolic} on $U$ if each pair $(g_k, T_k)$ is pseudo-hyperbolic on $U$ with respect to the same splitting $\reals^d = \Ecs \oplus \Eu$ and the same norms.
\end{definition}

The next theorem establishes the existence of a center-stable set for such sequences of globally jointly pseudo-hyperbolic pairs.  
Specifically, if the associated constants satisfy a suitable non-summability condition, then one can construct a center-stable set that captures all trajectories which eventually remain in a neighborhood of the fixed origin.

\begin{theorem} \label{theorem:local-csmt-generic}
    Let $\big( (g_k, T_k \colon \reals^d \to \reals^d) \big)_{k \geq 0}$ be a jointly \emph{globally} pseudo-hyperbolic sequence of pairs $(g_k, T_k) \in \mathrm{PH}(\mu_k, \lambda_k, \varepsilon_k; \reals^d, \Ecs \oplus \Eu)$ with constants satisfying $\sum_{k=0}^\infty \frac{\varepsilon_k}{\mu_k - 2 \varepsilon_k} = \infty.$
    Then, there exist Lipschitz functions $\varphi_0, \varphi_1, \dots \colon \Ecs \to \Eu$ such that,
    if the sequence $(x_k)_{k \geq 0}$ is bounded and satisfies $x_{k+1} = g_k(x_k)$ for all $k \geq \bar{k}$, then $x_k \in \graph(\varphi_k)$ for all $k \geq \bar{k}$.\footnote{As per our notation~\eqref{eq:graphnotation}, $x \in \graph(\varphi)$ means that if $x = y + z$ with $y \in \Ecs$ and $z \in \Eu$, then $z = \varphi(y)$.}
\end{theorem}

\begin{remark}
    The assumption of a common splitting $\Ecs \oplus \Eu$ is necessary in general (and is fulfilled in the optimization settings we consider later on).
    To see this, consider the following linear maps with different invariant splittings:
    \begin{align*}
        g_1 =
            \begin{pmatrix}
            0 & 0\\[2pt]
            -\frac15 & 2
            \end{pmatrix},
        \qquad \qquad
        g_2 =
            \begin{pmatrix}
            198 & \frac15\\[2pt]
            0 & 2
            \end{pmatrix}.
    \end{align*}
    The origin is a strictly unstable fixed point for each map individually, yet 
    \begin{align*}
        (g_1 \circ g_2 \circ g_2)^2 = 0.
    \end{align*}
    This shows that, without a common invariant splitting, a set of positive measure of initial conditions can eventually be mapped to an unstable fixed point.
\end{remark}

\noindent We postpone the proof of Theorem~\ref{theorem:local-csmt-generic} to Section~\ref{section:graph-transform-nonautonomous}.
Observe that this is a \emph{global} result: the pseudo-hyperbolicity
assumptions, and in particular the small-Lipschitz control in~\aref{assumption:small-lipschitz} is required to hold on all of \(\reals^d\).
In many applications, this control is available only \emph{locally}, in a neighborhood of a fixed point.
In the next step, we remove this globality requirement by showing how to reduce the local setting
to the global one without changing the dynamics near the fixed point.

As a minor notational point, the fixed point need not be at the origin.
Accordingly, if \(x^\ast\) is such that \(g_k(x^\ast)=x^\ast\) for all \(k\), we consider the shifted maps
\begin{align*}
    (g_k-x^\ast)(x) \coloneqq g_k(x+x^\ast)-x^\ast,
\end{align*}
restricted to an open ball $U := B_r(0)$ around the fixed point $0 = (g_k - x^\ast)(0)$.

\begin{definition} \label{def:unstable-fixed-point}
    A point $x^\ast \in \reals^d$ is a \emph{non-uniformly pseudo-hyperbolic (NPH) unstable fixed point} of the non-autonomous system given by the maps $\big( g_k \colon \reals^d \to \reals^d \big)_{k \geq 0}$ if there exist
    a splitting $\reals^d = \Ecs \oplus \Eu$ with $\dim(\Eu) \geq 1$,
    a family of linear maps $\big( T_k \colon \reals^d \to \reals^d \big)_{k \geq 0}$,
    a radius $r > 0$,
    an index $K \geq 0$ and
    a sequence of positive constants $( \mu_k, \lambda_k, \varepsilon_k )_{k \geq K}$ such that
    \begin{itemize}
        \item $(g_k - x^\ast, T_k) \in \mathrm{PH}\Big(\mu_k, \lambda_k, \frac{\varepsilon_k}{4}; B_r(0), \Ecs \oplus \Eu\Big)$ (w.r.t.\ the same norms) for all $k \geq K$,
        \item $\varepsilon_k < (\mu_k - \lambda_k) / 4$ for all $k \geq K$, and
        \item $\sum_{k=K}^\infty \varepsilon_k / (\mu_k - 2 \varepsilon_k) = + \infty$.
    \end{itemize}
    In particular, the pairs $(g_k - x^\ast, T_k)$ are (locally) jointly pseudo-hyperbolic as per Definition~\ref{def:jointlyPHonU}.
\end{definition}

\begin{remark}
    If the system is $C^1$ and autonomous ($g_k \equiv g$) then NPH unstable fixed points coincide with the
    usual notion of strictly unstable fixed points.  
    Indeed, we can take $T_k = \D g(x^\ast)$ and the pseudo-hyperbolicity condition
    corresponds to the usual center-stable/unstable splitting of the linearization at $x^\ast$. 
    In this case, the non-summability is satisfied easily since the sequence $( \mu_k, \lambda_k, \varepsilon_k )_{k \geq 0}$ can be taken constant.
\end{remark}

To study the local stable set of an NPH unstable fixed point $x^\ast$, we smoothly modify the maps
$g_k$ outside a small neighborhood of $x^\ast$ so that the resulting sequence becomes jointly \emph{globally} pseudo-hyperbolic.
This allows us to apply Theorem~\ref{theorem:local-csmt-generic} and construct the local center-stable set as the union of the graphs of the resulting Lipschitz functions $\varphi_k$.

\begin{theorem} [non-autonomous center-stable set theorem] \label{theorem:csmt-general}
    Let $x^\ast \in \reals^d$ be an NPH unstable fixed point of a non-autonomous dynamical system $\big( g_k \colon \reals^d \to \reals^d \big)_{k \geq 0}$.
    Then, there exist
    \begin{itemize}
        \item an open neighborhood $B$ of $x^\ast$,
        \item a measure zero set $\Wcsloc$ in $\reals^d$,
    \end{itemize}
    and an integer $K \geq 0$
    with the following property:
    if $( x_k )_{k \geq 0}$ is a sequence such that, for some $\bar{k} \geq K$, we have $x_{k+1} = g_k(x_k)$ and $x_k \in B$ for all $k \geq \bar{k}$, then 
    $x_k \in \Wcsloc$ for all $k \geq \bar{k}$.
\end{theorem}
\begin{proof}
    Assume without loss of generality that $x^\ast = 0$.
    Since $x^\ast$ is an NPH unstable fixed point, there exist a family of linear maps $\big( T_k \colon \reals^d \to \reals^d  \big)_{k \geq 0}$, a constant $r > 0$, an index $K \geq 0$ and a sequence $( \mu_k, \lambda_k, \varepsilon_k )_{k \geq K}$ such that (among other things) $\Lip((g_k - T_k)\vert_{B_r(0)}) \leq \varepsilon_k / 4$ for all $k \geq K$.
    Using a smooth transition function, we modify all maps $(g_k)_{k \geq K}$ outside a \emph{common} neighborhood of the origin so that the small-Lipschitz condition holds globally---see Lemma~\ref{lemma:local-to-global-small-lipschitz}: it provides maps $\tilde{g}_k \colon \reals^d \to \reals^d$ such that $\tilde{g}_k = g_k$ on $B \coloneqq B_{r/2}(0)$ and $\Lip(\tilde{g}_k - T_k) \leq \varepsilon_k$ for all $k \geq K$.
    
    Furthermore, the dynamics of $g_k$ and $\tilde{g}_k$ coincide locally.
    Indeed, consider a sequence $(x_k)_{k \geq 0}$ such that, for some $\bar{k} \geq K$, we have $x_{k + 1} = g_k(x_k)$ and $x_k \in B$ for all $k \geq \bar{k}$.
    Then, $x_{k+1} = g_k(x_k) = \tilde{g}_k (x_k) \in B$ for all $k \geq \bar{k}$.
    Thus, it suffices to show the result for the dynamical system corresponding to the $\tilde{g}_k$ maps.
    Observe that together with the other conditions from Definition~\ref{def:unstable-fixed-point}, the global
    Lipschitz control over $\tilde{g}_k - T_k$ implies that 
    \begin{align*}
        (\tilde{g}_k, T_k) \in \mathrm{PH}\Big( \mu_k, \lambda_k, \varepsilon_k; \reals^d, \Ecs \oplus \Eu \Big).
    \end{align*}
    By Theorem~\ref{theorem:local-csmt-generic} applied to the pairs $(\tilde{g}_k, T_k)_{k \geq K}$, there exists a sequence of Lipschitz functions $\varphi_K, \varphi_{K+1}, \dots \colon \Ecs \to \Eu$ with the following property:
    if the sequence $(x_k)_{k \geq 0}$ satisfies $x_k \in B$ and $x_{k+1} = g_k(x_k)$ for all $k \geq \bar{k} \geq K$, then $x_k \in \graph(\varphi_k)$ for all $k \geq \bar{k}$.
    Thus, the set 
    \begin{align*}
        \Wcsloc \coloneqq \bigcup_{k \geq K} \graph(\varphi_k)
    \end{align*}
    satisfies the desired invariance property.
    Moreover, each $\varphi_k \colon \Ecs \to \Eu$ is Lipschitz and $\dim(\Eu) \geq 1$, so each graph has measure zero in $\reals^d$.
    Hence, $\Wcsloc$ also has measure zero.
\end{proof}

\subsection{A global result}

With the local characterization of trajectories near an NPH unstable fixed point in hand, we now extend the avoidance guarantee globally.
The main ingredient is the following property.
Let $\mu$ denote the (standard) $d$-dimensional Lebesgue measure on $\reals^d$.

\begin{definition} \label{def:luzin}
    A map $g \colon \reals^d \rightarrow \reals^d$ has the 
    \emph{Luzin $N^{-1}$ property} if 
    \begin{equation*}
        \text{for all } E \subseteq \reals^d, \qquad\qquad
        \mu(E) = 0 
        \quad \implies \quad
        \mu(g^{-1}(E)) = 0.
    \tag{Luzin $N^{-1}$}
    \end{equation*}
\end{definition}
Although the Luzin $N^{-1}$ property may seem hard to check at first, a classical lemma provides a simpler characterization for continuously differentiable $g$.

\begin{lemma} \citep{ponomarev1987submersions} \label{lemma:luzin-verify}
    If $g \colon \reals^d \rightarrow \reals^d$ is a $C^1$ map, then $g$ has the Luzin $N^{-1}$ property if and only if $\rank \D g(x) = d$ for almost all $x \in \reals^d$.    
\end{lemma}

\begin{remark} \label{remark:gd-restricted-step-always-luzin}
    The Luzin $N^{-1}$ property is fairly common in optimization.
    For example, if $f \colon \reals^d \to \reals$ is $C^2$ and has $L$-Lipschitz gradient, then the gradient descent iteration map $g_{\alpha}(x) = x - \alpha \nabla f(x)$ satisfies the Luzin $N^{-1}$ property for all $\alpha \in (0, 1/L)$, because $\D g_{\alpha}(x) = I - \alpha \nabla^2 f(x)$ is invertible for \emph{all} $x$.
    In Sections~\ref{sec:luzin-gd} and~\ref{subsec:RGD}, we exploit the extra leeway (``for \emph{almost} all $x$'') to remove the Lipschitz assumption and allow arbitrarily large step sizes.
\end{remark}

When all the maps $g_k$ satisfy the Luzin $N^{-1}$ property, we can extend the unstable fixed point avoidance result globally.
Sections~\ref{section:gd}--\ref{section:pp} provide applications of this theorem.

\begin{theorem} [Global NPH unstable fixed point avoidance] \label{thm:global-avoidance-nonautonomous}
    If each map $g_k$ of the non-autonomous dynamical system $\big( g_k \colon \reals^d \to \reals^d \big)_{k \geq 0}$ has the Luzin $N^{-1}$ property, then the system avoids its set of NPH unstable fixed points (see Definitions~\ref{def:stable-set} and~\ref{def:unstable-fixed-point}).
\end{theorem}
\begin{proof}
    The proof is an immediate adaptation of the now standard arguments in \citep[Thm.~2]{lee2019strictsaddles}.
    To handle our generalized setup, it relies on the new Theorem~\ref{theorem:csmt-general} which itself relies on Theorem~\ref{theorem:local-csmt-generic} (proved in Sections~\ref{section:graph-transform-autonomous}--\ref{section:graph-transform-nonautonomous}).

    Let $x^\ast$ be an NPH unstable fixed point of $( g_k )_{k \geq 0}$.
    Applying Theorem~\ref{theorem:csmt-general}, we obtain an integer $K \geq 0$, an open neighborhood $B({x^\ast})$ of $x^\ast$ and a measure zero set $\Wcsloc(x^\ast)$ in $\reals^d$ satisfying the following:
    for a trajectory $(x_k)_{k \geq 0}$,
    if there exists $k_0 \geq K$ such that $x_k \in B({x^\ast})$ for all $k \geq k_0$, then $x_k \in \Wcsloc(x^\ast)$ for all $k \geq k_0$.

    Accordingly, if $x_0$ is such that the trajectory $(x_k)_{k \geq 0}$ eventually enters and never again exits $B({x^\ast})$,
    then there exists some $k \geq K$ such that $x_{k}$ is in $\Wcsloc(x^\ast)$.
    Since $x_{k} = \Phi_{k}(x_{0})$ with $\Phi_{k} \coloneqq g_{k-1} \circ \cdots \circ g_0$, it follows that
    \begin{align} \label{eq:preimage-wcsloc}
        x_0 \in \bigcup_{k \geq K} \Phi_{k}^{-1} (\Wcsloc(x^\ast)).
    \end{align}
    Let $S$ denote the set of all NPH unstable fixed points of $( g_k )_{k \geq 0}$.
    Applying the argument from above at each $x^\ast \in S$, we obtain sets $\{ B({x^\ast}) \}_{x^\ast \in S}$ which form a potentially uncountable open cover for $S$.
    By Lindel\"{o}f's lemma, we can extract a countable set $\{x_i^\ast\}_{i \geq 0}$ such that the associated open sets $\{ B({x_i^\ast}) \}_{i \geq 0}$ cover $S$.
    
    Now, consider $x_0 \in \reals^d$ such that the trajectory $(x_k)_{k \geq 0}$ converges to an (arbitrary) NPH unstable fixed point $x^\ast \in S$.
    Then, there exists some index $i \geq 0$ such that $x^\ast \in B({x_i^\ast})$.
    Since the trajectory eventually enters $B({x_i^\ast})$ and never exits it again, from Eq.~\eqref{eq:preimage-wcsloc}, we obtain that $x_0$ is in the set 
    \begin{align*}
        W \coloneqq \bigcup_{i \geq 0} \; \bigcup_{k \geq K} \Phi_{k}^{-1} (\Wcsloc(x_i^\ast)). 
    \end{align*}
    As all maps $g_k$ have the Luzin $N^{-1}$ property, the maps $\Phi_{k}$ also have the Luzin $N^{-1}$ property for all $k \geq 0$.
    Then, since the sets $\Wcsloc(x_i^\ast)$ all have measure zero in $\reals^d$, it follows that 
    $W$ has measure zero as a countable union of measure zero sets.
\end{proof}

\section{The graph transform method for autonomous systems} \label{section:graph-transform-autonomous}

In this section, we review some well-known ingredients of the graph transform method for autonomous systems ($x_{k+1} = g(x_k)$): we shall repurpose them in Section~\ref{section:graph-transform-nonautonomous} for non-autonomous systems.
For completeness, we include proofs in Appendix~\ref{appendix:graph-transform-proofs} based on~\citep{hirsch1977invariant}.

Let $\Ecs, \Eu \subseteq \reals^d$ be complementary subspaces such that $\reals^d = \Ecs \oplus \Eu $, that is, 
each $x \in \reals^d$ splits uniquely as $x = y + z$ with $y \in \Ecs$ and $z \in \Eu$.  
Let $\pcs \colon \reals^d \to \Ecs$ and $\pu \colon \reals^d \to \Eu$ denote the corresponding projectors, so that
$y = \pcs(x)$ and $z = \pu(x)$.
Given a map $H \colon \reals^d \to \reals^d$, for ease of notation we define 
$H_{\mathrm{cs}} \coloneqq \pcs \circ H$ and $H_{\mathrm{u}} \coloneqq \pu \circ H$.

The subspaces $\Ecs$ and $\Eu$ each have their norm, and we equip $\Rd$ with the induced max-norm~\eqref{eq:Rd-norm}.

The (inverse) graph transform method constructs a center-stable manifold by working within a class of candidate graphs:
we consider certain maps $\varphi \colon \Ecs \to \Eu$ and identify each candidate set with $\graph(\varphi) \subset \Ecs \oplus \Eu$.
Then, we consider a graph transform operator that maps $\graph(\varphi)$ to a new graph, which represents (locally) a kind of inverse image of $\graph(\varphi)$ under the dynamics.
The center-stable manifold is the limit of this sequence of graphs.

\subsection{A graph invariance property}

For the rest of this section, we fix a pseudo-hyperbolic pair $(g, T) \in \mathrm{PH}(\mu, \lambda, \varepsilon; \reals^d, \Ecs \oplus \Eu)$ as per Definition~\ref{def:pseudo-hyperbolicity}.
\emph{Given} some Lipschitz function $\varphi \colon \Ecs \to \Eu$, our goal is to \emph{construct} $\tilde{\varphi} \colon \Ecs \to \Eu$ such that the graph of $\tilde{\varphi}$ is mapped by $g$ into the graph of $\varphi$, that is,\footnote{In this sense, $\graph(\tilde{\varphi})$ is a kind of inverse image of $\graph(\varphi)$ under the dynamical system~$g$, even though~$g$ itself may not be invertible.}
\begin{align} \label{eq:graph-invariance}
    x \in \graph(\tilde{\varphi}) \quad \implies \quad g(x) \in \graph(\varphi).
\end{align}
As before (recall Eq.~\eqref{eq:graphnotation}), $\graph({\varphi})$ denotes the set of points $x \in \reals^d$ that admit the decomposition $x = y + {\varphi}(y)$ with $y \in \Ecs$.
We also write this as $x = (y, {\varphi}(y))$.
Equivalently, these are points satisfying $\pu(x) = {\varphi}(\pcs(x))$.

Using this notation, Eq.~\eqref{eq:graph-invariance} can be restated as follows: each point in $\graph(\tilde{\varphi})$ is of the form $(y, \tilde{\varphi}(y))$ for some $y \in \Ecs$, and we require $g(y, \tilde{\varphi}(y))$ to be in $\graph(\varphi)$, that is,
\begin{align} \label{eq:graph-invariance-expanded}
    \pu \big( g(y, \tilde{\varphi}(y)) \big) =
    \varphi \big(\pcs( g( y,\tilde\varphi(y) )) \big). 
\end{align}
For each $y \in \Ecs$, our goal is to determine $\tilde{\varphi}(y)$ 
satisfying the relationship above.

We proceed by rewriting condition~\eqref{eq:graph-invariance-expanded} as a fixed-point equation and showing that it admits a unique solution.
To do so, decompose $\pu \circ g = \pu \circ (g-T) + \pu \circ T$.
By the invariance of the splitting of $\reals^d$ under $T$, for any $(y, z) \in \reals^d$ we have $\pu (T(y,z)) = T\vert_{\Eu}\,z$, where $T\vert_{\Eu} \colon \Eu \to \Eu$ is the restriction of $T$ to the unstable subspace.\footnote{Not to be confused with $\Tu = \pu \circ T$, which is defined on all of $\reals^d$.}
Then, writing $z := \tilde\varphi(y)$, condition~\eqref{eq:graph-invariance-expanded} becomes 
\begin{align*}
    T\vert_{\Eu}\, z =
    \varphi\big( \pcs(g(y,z)) \big) - \pu \big( (g-T)(y,z) \big).
\end{align*}
Since $T|_{\Eu}\colon \Eu\to \Eu$ is invertible by~\aref{assumption:pseudo-hyperbolic}, it follows that $z$ must be a fixed point of the \emph{auxiliary map} $h_y^\varphi\colon \Eu\to \Eu$ defined by
\begin{align} \label{eq:h-phi-y-definition}
    h_y^\varphi(z) = 
    \big( T\vert_{\Eu} \big)^{-1} 
    \Big( \varphi\big( \gcs(y,z) \big) - \bigl((g_u-T_u)(y,z)\bigr) \Big),
\end{align}
where we use the notation $\gcs = \pcs \circ g$ and $\gu = \pu \circ g$.

Lemma~\ref{lemma:h_y_phi_lipschitz} below implies for each $y \in \Ecs$ that the map $h_y^\varphi$ is a contraction.
Consequently, it has a unique fixed point which we denote by $\Fix(h_y^\varphi)$.
Then, the function
\begin{align*}
    \tilde{\varphi}(y) \coloneqq \Fix(h_y^\varphi)
\end{align*}
is well defined, and it satisfies our requirement in Eq.~\eqref{eq:graph-invariance} (by design).

Since we later study the regularity of $\tilde{\varphi}$, it is useful to record already now that, for fixed $z$, the map $y \to h_y^\varphi(z)$ is Lipschitz.
This observation will allow us to show that if $\varphi$ belongs to a given Lipschitz class, then so does $\tilde{\varphi}$.

\begin{lemma} \label{lemma:h_y_phi_lipschitz}
    Let $(g, T) \in \mathrm{PH}(\mu, \lambda, \varepsilon; \reals^d, \Ecs \oplus \Eu)$ and let $\varphi \colon \Ecs \to \Eu$ satisfy $\Lip(\varphi) \leq 1$ and $\varphi(0) = 0$.
    Then, the associated auxiliary map $\hyvarphi$~\eqref{eq:h-phi-y-definition} satisfies:
    \begin{itemize}
        \item For all $y \in \Ecs$, it holds that $\Lip(z \mapsto \hyvarphi(z)) \leq 2 \varepsilon / \mu < 1$. 
        \item For all $z \in \Eu$, it holds that $\Lip(y \mapsto \hyvarphi(z)) \leq (\lambda + 2 \varepsilon) / \mu $.
    \end{itemize}
\end{lemma}
\noindent
The proof is in Appendix~\ref{sec:proof_lemma:h_y_phi_lipschitz}.

\subsection{Function space setup for the graph transform} \label{section:function-space}

For a function $\varphi \colon \Ecs \to \Eu$, define the quantity 
\begin{align} \label{eq:func-norm}
    \| \varphi \| \coloneq \sup_{y \neq 0} \frac{\| \varphi(y) \|}{\| y \|}.
\end{align}
Consider the vector space 
\begin{align*}
    \calF = \{ \varphi \colon \Ecs \to \Eu \mid \varphi(0) = 0 \textrm{ and } \| \varphi \| < \infty \}.   
\end{align*}
On $\calF$, the quantity $\| \cdot \|$ defines a norm.
\ifbool{BanachAppendix}{
   Equipped with this norm, $\calF$ is a complete metric space (that is, a Banach space; see Proposition~\ref{prop:banach-space}).
}{
   Equipped with this norm, $\calF$ is a complete metric space (that is, a Banach space).
}%
We next focus on the subset
\begin{align} \label{eq:func-norm-1}
    \calF_1 = \{ \varphi \in \calF \mid \Lip(\varphi) \leq 1 \}.
\end{align}
It is closed in $\calF$. 
For a pseudo-hyperbolic pair $(g, T)$, define the \emph{graph transform map} 
\begin{align} \label{eq:graph-transform-map}
    \Gamma \colon \calF_1 \to \calF_1, \quad\quad
    (\Gamma \varphi)(y) = \Fix(h_y^\varphi),
\end{align}
where $h_y^\varphi$ is the auxiliary map associated with $\varphi$ at $y$~\eqref{eq:h-phi-y-definition} and $\Fix(h_y^\varphi)$ denotes the unique fixed point of $h_y^\varphi$ (well defined owing to Lemma~\ref{lemma:h_y_phi_lipschitz}).

By design (recall Eq.~\eqref{eq:graph-invariance}), the graph transform $\Gamma$ satisfies the invariance property
\begin{align} \label{eq:graph-invariance-gamma}
    x \in \graph(\Gamma \varphi) \quad \implies \quad g(x) \in \graph(\varphi).
\end{align}
The following lemma states that $\Gamma$ indeed maps $\calF_1$ into $\calF_1$ and that it is a contraction in the norm from Eq.~\eqref{eq:func-norm}.
The proof follows the classical arguments in~\citep[Thm.~5.1]{hirsch1977invariant}: see Appendix~\ref{subsec:gamma-lipschitz}.

\begin{lemma}  \label{lemma:gamma-lipschitz}
    Let $(g, T) \in \mathrm{PH}(\mu, \lambda, \varepsilon; \reals^d, \Ecs \oplus \Eu)$ and let $\Gamma \colon \calF_1 \to \calF_1$ be the associated graph transform map~\eqref{eq:graph-transform-map}.
    Then, $\Gamma$ is well defined and it is a contraction with respect to the norm defined in~\eqref{eq:func-norm}, with Lipschitz constant
    $ \Lip(\Gamma) \leq (\lambda + \varepsilon) / (\mu - 2 \varepsilon) < 1 $. 
\end{lemma}

Since $\calF_1$ is complete (as a closed subset of $\calF$), the fact that $\Gamma$ is a contraction implies the existence of a (unique) function $\varphi \in \calF_1$ such that $\Gamma \varphi = \varphi$.
Combined with Eq.~\eqref{eq:graph-invariance-gamma}, this reveals what would be the sought invariance property in the \emph{autonomous} setting:
\begin{align*}
    x \in \graph(\varphi) \quad \implies \quad g(x) \in \graph(\varphi).
\end{align*}
In the non-autonomous setting, however, the situation changes: at each time step $k$, the map $g_k$ is different, so the associated graph transform $\Gamma_k$ also changes with $k$.
This is why in Section~\ref{section:graph-transform-nonautonomous} we repurpose these constructions.

\subsection{The potential}

For any function $\varphi \in \calF_1$, define the potential function $V_\varphi \colon \reals^d \to \reals$ by
\begin{align} \label{eq:potential-def}
    V_\varphi(x) = \| \pu(x) - \varphi(\pcs(x)) \|. 
\end{align}
This quantity measures how much a point $x \in \reals^d$ departs from the graph of $\varphi$.
Recall the invariance property in Eq.~\eqref{eq:graph-invariance-gamma}.
The next lemma shows that, unless $x$ lies on the graph of $\Gamma \varphi$, the potential at $g(x)$ is strictly larger than the potential at $x$, albeit with the former measured with respect to $V_\varphi$ and the latter with respect to $V_{\Gamma \varphi}$.
The proof is included in Appendix~\ref{sec:proof-lemma-potential-growth}.

\begin{lemma} \label{lemma:potential-growth}
    Let $(g, T) \in \mathrm{PH}(\mu, \lambda, \varepsilon; \reals^d, \Ecs \oplus \Eu)$ and let $\Gamma$ be the associated graph transform map~\eqref{eq:graph-transform-map}.
    Then, for any $\varphi \in \calF_1$ and any $x \in \reals^d$, 
    \begin{equation*}
        V_\varphi(g(x)) \geq (\mu - 2 \varepsilon) V_{\Gamma \varphi}(x),
    \end{equation*}
    where the multiplier satisfies $\mu - 2\varepsilon > 1$.
\end{lemma}

In the autonomous setting, this is most interesting when $\varphi$ is the fixed point of $\Gamma$, as then we can say more succinctly: the potential $V_\varphi$ along a trajectory $(x_k)_{k \geq 0}$ is either constant (zero), or it diverges exponentially.
If the trajectory itself remains bounded, then continuity of $V_\varphi$ ensures $( V_\varphi(x_k) )_{k \geq 0}$ is bounded as well, which forces the entire sequence $(x_k)$ to be on $\graph(\varphi)$.
This is how one can conclude that $\graph(\varphi)$ is a center-stable set for the unstable fixed point (the origin).
This graph has measure zero since $\varphi$ is Lipschitz continuous.

\section{Exploiting the graph transform with non-autonomous systems} \label{section:graph-transform-nonautonomous}

We are now ready to prove Theorem~\ref{theorem:local-csmt-generic} by adapting the classical graph transform argument from the previous section.
Recall the norm on $\reals^d$ is as in Eq.~\eqref{eq:Rd-norm}.

Importantly, in the non-autonomous setting the iteration map $g_k$ varies with $k$, so there is no single transform map $\Gamma$ with a fixed point $\varphi$.
Instead, each pseudo-hyperbolic pair $(g_k,T_k)$ induces its own graph transform $\Gamma_k$~\eqref{eq:graph-transform-map}.
Accordingly, we aim to construct a sequence $(\varphi_k)_{k\ge 0}$ of Lipschitz functions satisfying the invariance property\footnote{Equivalently, $g_k \big(\graph(\varphi_{k}) \big)\subseteq \graph(\varphi_{k+1})$.}
\begin{align} \label{eq:graphinvarianceproofmainthm}
    x \in \graph(\varphi_{k}) \quad \implies \quad g_k(x) \in \graph(\varphi_{k+1}).
\end{align}
Comparing with~\eqref{eq:graph-invariance-gamma}, we see that this requires $\Gamma_k \varphi_{k+1} = \varphi_k$: this cannot be used directly to build the sequence $(\varphi_k)$, which is why the proof proceeds in a different manner.

With this in mind, let us start the proof.
The space of functions 
$\calF_1 = \{ \varphi \colon \Ecs \to \Eu \mid \varphi(0) = 0, \| \varphi \| < \infty, \; \Lip(\varphi) \leq 1 \}$
is the same as before (Eq.~\eqref{eq:func-norm-1}).
For each pseudo-hyperbolic pair $(g_k, T_k)$, let $\Gamma_{k} \colon \calF_1 \to \calF_1$ be the associated graph transform map as defined in Eq.~\eqref{eq:graph-transform-map}.
From Lemma~\ref{lemma:gamma-lipschitz}, each $\Gamma_k$ is Lipschitz with constant $\Lip(\Gamma_k) \leq (\lambda_k + \varepsilon_k) / (\mu_k - 2 \varepsilon_k)$.

\paragraph{General potential growth inequality.}

Fix some $\varphi \in \calF_1$, arbitrary for now.
If $(x_k)_{k \geq 0}$ is a sequence such that $x_{k+1} = g_k(x_k)$ for all $k \geq \bar{k}$ (for some $\bar{k} \geq 0$), then unrolling the potential growth inequality from Lemma~\ref{lemma:potential-growth} yields
\begin{align*}
    V_\varphi(x_{k + 1}) = 
    V_\varphi(g_k(x_k))
    & \geq (\mu_k - 2 \varepsilon_k) \; V_{\Gamma_k \varphi} (x_k) \\ 
    & = (\mu_k - 2 \varepsilon_k) \; V_{\Gamma_k \varphi} (g_{k-1}(x_{k-1})) \\ 
    & \geq (\mu_k - 2 \varepsilon_k) (\mu_{k-1} - 2 \varepsilon_{k-1}) \; V_{(\Gamma_{k-1} \circ \Gamma_{k}) \varphi} (g_{k-2}(x_{k-2})) \\ 
    & \geq \; \cdots \\ 
    & \geq  (\mu_k - 2 \varepsilon_k) (\mu_{k-1} - 2 \varepsilon_{k-1}) \cdots (\mu_{\bar{k}} - 2 \varepsilon_{\bar{k}}) \; V_{(\Gamma_{\bar{k}} \circ \dots \circ \Gamma_{k-1} \circ \Gamma_k) \varphi}(x_{\bar{k}}).
\end{align*}
Since this inequality holds for all $\varphi \in \calF_1$, it holds in particular for the zero function $0 \in \calF_1$.
Then, for all $k \geq \bar{k}$,
\begin{align} \label{eq:pu_bounded}
    \| \pu(x_{k+1}) \| = 
    V_0 (x_{k+1}) 
    \geq (\mu_k - 2 \varepsilon_k) \cdots (\mu_{\bar{k}} - 2 \varepsilon_{\bar{k}}) V_{\varphi_{\bar{k}, k}}(x_{\bar{k}}),
\end{align}
where we defined the following functions (independent of the sequence $(x_k)_{k \geq 0}$):
\begin{align}
    \varphi_{\bar{k}, k} \coloneqq (\Gamma_{\bar{k}} \circ \cdots \circ \Gamma_{k-1} \circ \Gamma_k)(0) && \textrm{for all } k \geq \bar{k}.
    \label{eq:varphibarkk}
\end{align}
These are indeed in $\calF_1$ because each map $\Gamma_i$ maps $\calF_1$ to $\calF_1$.

\paragraph{First use of the non-summability assumption.}

Let us show that the product of factors $(\mu_i - 2 \varepsilon_i)$ in~\eqref{eq:pu_bounded} diverges to infinity.
Recall from Definition~\ref{def:pseudo-hyperbolicity} that $1 \leq \lambda_k < \mu_k$ and $\mu_k - \lambda_k > 4\varepsilon_k$ (from \aref{assumption:small-lipschitz}).
Combined, these yield $\mu_k - 2 \varepsilon_k - 1 > 2 \varepsilon_k > 0$ for all $k$.
Used together with the non-summability assumption and $\log(t) \geq \frac{t-1}{t}$ for all $t > 1$, we obtain
\begin{align} \label{eq:nonsummabilityfirstuse}
    \sum_{k=\bar{k}}^\infty \log(\mu_k - 2 \varepsilon_k) \geq 
    \sum_{k=\bar{k}}^\infty \frac{\mu_k - 2 \varepsilon_k - 1}{\mu_k - 2 \varepsilon_k} \geq 
        \sum_{k=\bar{k}}^\infty \frac{\varepsilon_k}{\mu_k - 2 \varepsilon_k} = 
    \infty.
\end{align}
Taking exponentials, we find $\prod_{k = \bar{k}}^\infty (\mu_k - 2 \varepsilon_k) = \infty$ as announced.

\paragraph{Implications for bounded trajectories.}

If the sequence $( x_k )_{k \geq 0}$ is bounded, there exists $r \geq 0$ such that $\|\pu(x_{k+1})\| \leq \|x_{k+1}\| \leq r$ for all $k$.
Combined with~\eqref{eq:pu_bounded} and~\eqref{eq:nonsummabilityfirstuse},
this implies
\begin{align}
    V_{\varphi_{\bar{k}, k}}(x_{\bar{k}}) = \| \pu(x_{\bar{k}}) - \varphi_{\bar{k}, k}(\pcs(x_{\bar{k}})) \| \xrightarrow[k \to \infty]{} 0. 
    \label{eq:limitVforboundedseq}
\end{align}
Thus, $\pu(x_{\bar{k}}) = \lim_{k \to \infty} \varphi_{\bar{k}, k} (\pcs(x_{\bar{k}}))$.
This brings us to the central question: for each fixed $\bar{k} \geq 0$, does the sequence of functions $(\varphi_{\bar{k}, k})_{k \geq \bar{k}}$ have a (nice) limit?
It does (as we show next), upon which it follows from~\eqref{eq:limitVforboundedseq} that $x_{\bar{k}}$ is on its graph.

\paragraph{Building the functions $\varphi_0, \varphi_1, \ldots$}

Tentatively define $\varphi_{\bar{k}} \coloneqq \lim_{k \to \infty} \varphi_{\bar{k}, k} $.
It remains to show that $\varphi_{\bar{k}}$ exists and that it belongs to $\calF_1$.
To this end, we show the sequence $( \varphi_{\bar{k}, \bar{k} + t} )_{t \geq 0} $ is Cauchy in $\calF_1$.
Let $\bar{k} \leq i < j$.
By definition of $\varphi_{\bar{k}, i}$, we have that 
\begin{align} \label{eq:varphi-sequence-cauchy}
    \| \varphi_{\bar{k}, j} - \varphi_{\bar{k}, i} \|
    &= \| ( \Gamma_{\bar{k}} \circ \dots \circ \Gamma_{i} \circ \Gamma_{i+1} \circ \dots \circ \Gamma_j) (0) - ( \Gamma_{\bar{k}} \circ \dots \circ \Gamma_{i} ) (0) \| \notag \\ 
    &\leq \Lip \big( \Gamma_{\bar{k}} \circ \cdots \circ \Gamma_{i} \big) \; \| (\Gamma_{i+1} \circ \cdots \circ \Gamma_j) (0) \| \notag \\ 
    &\leq \Lip(\Gamma_{\bar{k}} ) \cdots \Lip(\Gamma_{i}) \; \| \varphi_{i+1, j} \| \notag \\ 
    &\leq \Lip(\Gamma_{\bar{k}} ) \cdots \Lip(\Gamma_{i}), 
\end{align}
where in the last step we used that $\varphi_{i+1, j}$ is in $\calF_1$, and that for any $\varphi \in \calF_1$ it holds that 
\begin{align*}
    \| \varphi \| = 
    \sup_{y \neq 0} \frac{\| \varphi(y) \|}{\| y \|} =
    \sup_{y \neq 0} \frac{\| \varphi(y) - \varphi(0) \|}{\| y \|} \leq
    \Lip(\varphi) \leq 1.
\end{align*}
Let us show that the right-hand side of Eq.~\eqref{eq:varphi-sequence-cauchy} goes to zero as $i \to \infty$.
Once we secure this, we obtain for each $\bar{k} \geq 0$ that the sequence $( \varphi_{\bar{k}, k} )_{k \geq \bar{k}}$ is Cauchy in a complete metric space, and thus its limit (which we preemptively called $\varphi_{\bar{k}}$) exists and belongs to $\calF_1$.

\paragraph{Second use of the non-summability assumption.}

By Lemma~\ref{lemma:gamma-lipschitz} and using $\log(t) \leq - (1-t)$ for all $t \in (0, 1)$ we have
\begin{align*}
    \log\!\Bigg( \prod_{k = \bar{k}}^i \Lip(\Gamma_k) \Bigg) \leq
    \sum_{k = \bar{k}}^i \log\!\Bigg( \frac{\lambda_{k} + \varepsilon_{k}}{\mu_{k} - 2 \varepsilon_{k} } \Bigg) \leq
    - \sum_{k = \bar{k}}^i \Bigg( 1 - \frac{\lambda_{k} + \varepsilon_{k}}{\mu_{k} - 2 \varepsilon_{k} } \Bigg) =
    - \sum_{k = \bar{k}}^i \frac{\mu_k - \lambda_k - 3 \varepsilon_k}{\mu_k - 2 \varepsilon_k}. 
\end{align*}
Again using $\mu_k - \lambda_k - 3 \varepsilon_k \geq \varepsilon_k$ from~\aref{assumption:small-lipschitz} as well as the non-summability assumption, this last series diverges to $ - \infty$ as $i \to \infty$.
Therefore, we have confirmed that $\prod_{k = \bar{k}}^\infty \Lip(\Gamma_k) = 0$.

\paragraph{Conclusion.}

For each $\bar{k} \geq 0$ we showed $\varphi_{\bar{k}} \coloneqq \lim_{k \to \infty} \varphi_{\bar{k},k}$ is a valid function in $\calF_1$: this provides us with a sequence of Lipschitz functions $\varphi_0, \varphi_1, \ldots$ as promised.
From~\eqref{eq:limitVforboundedseq}, we see that if $(x_k)_{k \geq 0}$ is any bounded sequence satisfying $x_{k+1} = g_k(x_k)$ for all $k \geq \bar{k}$, then $x_{\bar{k}}$ belongs to the graph of $\varphi_{\bar{k}}$.
Thus, the proof of Theorem~\ref{theorem:local-csmt-generic} is complete.

For good measure, let us connect this back to the driving intuition stated as the invariance condition~\eqref{eq:graphinvarianceproofmainthm}:
since $\varphi_{\bar{k}, k} = \Gamma_{\bar{k}} \varphi_{\bar{k}+1, k}$, passing to the limit $k \to \infty$ reveals $\varphi_{\bar{k}} = \Gamma_{\bar{k}} \varphi_{\bar{k}+1}$.
Thus,
\begin{align*}
    x \in \graph(\varphi_{k}) \iff x \in \graph(\Gamma_{k} \varphi_{k+1}) \overset{\eqref{eq:graph-invariance-gamma}}{\implies} g_{k}(x) \in \graph(\varphi_{k+1}) && \textrm{ for all } k \geq 0,
\end{align*}
as desired.

\begin{remark}
    By direct inspection of the two uses of the non-summability assumption in the proof of Theorem~\ref{theorem:local-csmt-generic}, it is clear that the assumption can be slightly relaxed to (both)
    \begin{align*}
        \sum_{k=0}^\infty \frac{\mu_k - 2 \varepsilon_k - 1 }{\mu_k - 2 \varepsilon_k} = + \infty 
        \quad\qquad \textrm{and} \quad\qquad 
        \sum_{k=0}^\infty \frac{\mu_k - \lambda_k - 3 \varepsilon_k}{\mu_k - 2 \varepsilon_k} = + \infty.
    \end{align*}
\end{remark}

\begin{remark}
    Theorem~\ref{theorem:local-csmt-generic} guarantees the existence of useful functions $\varphi_0, \varphi_1, \ldots$.
    These are Lipschitz, hence their graphs have zero measure: this is enough regularity for saddle avoidance purposes.
    For contrast, the \emph{autonomous} center-stable \emph{manifold} theorem further ensures that the limit function $\varphi$ (and hence its graph) is smooth~\citep[Thm.~5.1]{hirsch1977invariant}.
\end{remark}

\section{Gradient descent with vanishing step sizes} \label{section:gd}

As an application of Theorem~\ref{thm:global-avoidance-nonautonomous}, we study GD with various step-size schedules, that is, the dynamical system $(g_k)_{k \geq 0}$ with positive step sizes $(\alpha_k)_{k \geq 0}$ and
\begin{align}
    g_k(x) = x - \alpha_k \nabla f(x),
    \tag{GD}
    \label{eq:GD}
\end{align}
where $f$ is a $C^2$ cost function.
The following is the main result of this section.
It holds for a wider variety of step-size schedules (ruled by the conditions of Proposition~\ref{proposition:GD-admissible-corollary}): we single out a few important ones for illustration.
Notice $\nabla f$ need not be Lipschitz, and the step sizes may be arbitrarily large.

\begin{theorem} \label{thm:gd-vanishing-theorem}
    Let $f \colon \reals^d \to \reals$ be $C^2$.
    Consider~\eqref{eq:GD} with given $\alpha_0 > 0$ and one of the following step-size schedules for $k \geq 1$:
    \begin{itemize}
        \item Constant: $\alpha_k = \alpha_0$.
        \item Polynomial decay: $\alpha_k = \frac{\alpha_0}{(k+1)^\gamma}$ for some $\gamma \in (0, 1]$.
        \item Cosine decay: $\alpha_k = \frac{\alpha_0}{(k+1)^\gamma} \Big( 1 + \cos\!\big( \frac{\pi (k + \frac{1}{2})}{2 T + 1} \big) \Big)$ for some $\gamma \in (0, 1]$ and integer $T \geq 0$.
    \end{itemize}
    Then, for each fixed choice of the auxiliary parameters ($\gamma, T$) and for almost all $\alpha_0 > 0$, \eqref{eq:GD} avoids the strict saddle points of $f$.
\end{theorem}

We dedicate the rest of this section to a proof: first, in Section~\ref{sec:saddles-are-nph-gd}, we show that the strict saddles of $f$ correspond to NPH unstable fixed points of $(g_k)_{k \geq 0}$ (Definition~\ref{def:unstable-fixed-point}).
Then, in Section~\ref{sec:luzin-gd}, we make sure the iteration maps satisfy the Luzin $N^{-1}$ property (Definition~\ref{def:luzin}).
Both happen under certain conditions on the step-size schedule: we identify sufficient yet permissive ones.
Once both of these are secured, Theorem~\ref{thm:gd-vanishing-theorem} follows from Theorem~\ref{thm:global-avoidance-nonautonomous}.

\subsection{Strict saddles are NPH unstable} \label{sec:saddles-are-nph-gd}

We begin with a lemma establishing Lipschitz control of the iteration maps $g_k$ in a common neighborhood of the origin (later shifted to a strict saddle).
This is useful for claiming the uniform local Lipschitz control required by the definition of NPH unstable fixed points.

\begin{lemma} \label{lemma:local-lipschitz-control-GD}
    Let $f \colon \reals^d \to \reals$ be $C^2$ and
    consider \eqref{eq:GD} with positive step sizes $\alpha_0, \alpha_1, \ldots$.
    Then, for every $c > 0$, there exists $r > 0$ such that 
    \begin{align*}
        \Lip \Big( (g_k - \D g_k(0)) \vert_{B_r(0)} \Big) \leq c \alpha_k, 
        \qquad \textrm{ for all } k \geq 0.
    \end{align*}
    (This holds for all norms on $\reals^d$, only affecting the value of $r$.)
\end{lemma}
\begin{proof}
    Since $f$ is $C^2$, all $g_k$ are $C^1$ and we can compute 
    $\D g_k(x) = I_d - \alpha_k \nabla^2 f(x)$.
    Define $h_k(x) = g_k(x) - \D g_k(0)[x]$, which is also $C^1$ and satisfies 
    \begin{align*}
        \D h_k(x) = 
        \D g_k(x) - \D g_k(0) = 
        -\alpha_k (\nabla^2 f(x) - \nabla^2 f(0)).
    \end{align*}
    Then, we have that  
    \begin{align*}
        \Lip(h_k\vert_{B_r(0)}) 
        \leq \max_{\|x\| \leq r} \| \D h_k(x) \| 
        = \alpha_k \max_{\|x\| \leq r} \| \nabla^2 f(x) - \nabla^2 f(0) \| \eqqcolon \alpha_k \, \omega(r).
    \end{align*}
    The function $\omega \colon [0, \infty) \to \reals$ satisfies $\omega(0) = 0$ and it is continuous owing to the Maximum Theorem and continuity of $\nabla^2 f$~\citep[p.~116]{berge1963topological}.
    Therefore, given any $c > 0$, we can find $r > 0$ such that $\omega(r) \leq c$ (independent of $k$).
\end{proof}

Consider a strict saddle point $x^\ast$ of $f$.
We aim to show $x^\ast$ is NPH for \eqref{eq:GD}.
This notably requires us to split $\reals^d$ into a direct sum $\Ecs \oplus \Eu$ that (eventually) works for \emph{all} $g_k$.
If the eigenvalues of $\nabla^2 f(x^\ast)$ are $h_1, \ldots, h_d$, the eigenvalues of $\D g_k(x^\ast) = I_d - \alpha_k \nabla^2 f(x^\ast)$ are $1-\alpha_k h_1, \ldots, 1-\alpha_k h_d$.
From this, it is clear that negative eigenvalues ($h_i < 0$) always correspond to unstable directions ($1-\alpha_k h_i > 1$), for all $\alpha_k > 0$.
However, positive eigenvalues ($h_i > 0$) may be center-stable ($1-\alpha_k h_i \geq -1$) or unstable ($1-\alpha_k h_i < -1$), depending on how large $\alpha_k$ is.
For that reason, we need some condition on the step sizes that ensures the classification is eventually fixed.

The definition below formalizes that requirement, with an additional condition.
For $h_i < 0$, it is clear that $|1 - \alpha_k h_i| = 1 + \alpha_k |h_i|$ stays away from 1 proportionally to $\alpha_k$.
For $h_i > 0$, this is not automatic, hence we require it explicitly.
(This holds if $\alpha_k$ does not accumulate at $2/h_i$.)

\begin{definition} \label{def:admissiblesequence}
    A sequence of positive reals $\alpha_0, \alpha_1, \ldots$ is \emph{admissible} if, for each $h \in \Rd$, there exists $K \geq 0$ such that both of the following are true:
    \begin{enumerate}[label=(\roman*)]
        \item There exists a partition of the indices $\{1, \ldots, d\}$ as $\Ics \cup \Iu$ such that
        \begin{align*}
            |1 - \alpha_k h_i| \leq 1 && \textrm{ and } && |1 - \alpha_k h_j| > 1
        \end{align*}
        for all $k \geq K$, $i \in \Ics$ and $j \in \Iu$.
        \item There exists a constant $c > 0$ such that $|1 - \alpha_k h_j| \geq 1 + c \alpha_k$ for all $k \geq K$ and $j \in \Iu$.
    \end{enumerate}
\end{definition}

\begin{remark} \label{remark:admissible-steps}
    If the step sizes $\alpha_0, \alpha_1, \ldots$ are positive and either constant ($\alpha_k = \alpha_0$) or vanishing ($\alpha_k \to 0$), then the sequence $( \alpha_k )_{k \ge 0}$ is admissible.
    For the latter, the eventual partition associated to a vector $h \in \Rd$ is $\Ics = \{i : h_i \geq 0\}$ and $\Iu = \{j : h_j < 0\}$.
\end{remark}

The following lemma ensures that strict saddle points of a $C^2$ function are NPH unstable fixed points for gradient descent with admissible step sizes.

\begin{lemma} [Strict saddles are NPH unstable] \label{lemma:csmt-sequence-of-gradient-maps}
    Let $f \colon \reals^d \to \reals$ be $C^2$.
    If $( \alpha_k )_{k \geq 0}$ is admissible (Definition~\ref{def:admissiblesequence}) and $\sum_{k \geq 0} \alpha_k = \infty$,
    then every strict saddle point of $f$ is an NPH unstable fixed point for \eqref{eq:GD} (Definition~\ref{def:unstable-fixed-point}).
\end{lemma}
\begin{proof}
    Let $x^\ast$ be a strict saddle point of $f$.
    Without loss of generality, we can assume $x^\ast = 0$.
    Let $h_1, \ldots, h_d$ denote the eigenvalues of $\nabla^2 f(0)$.
    Since $\nabla f(0) = 0$ we have $g_k(0) = 0$ for all $k \geq 0$, so~\aref{assumption:fixed-origin} is satisfied for all $g_k$.
    Below, we let $T_k \coloneqq \D g_k(0) = I_d - \alpha_k \nabla^2 f(0)$ and check all other conditions of Definition~\ref{def:unstable-fixed-point}.

    \paragraph{Linearization and invariant splitting.}

    Since $(\alpha_k)_{k \geq 0}$ is admissible, there exist $K \geq 0$, $c > 0$ and a partition $\{1, \ldots, d\} = \Ics \cup \Iu$ such that, for all $k \geq K$,
    \begin{align*}
        |1 - \alpha_k h_i| \leq 1 \textrm{ for all } i \in \Ics && \textrm{ and } && |1 - \alpha_k h_j| \geq 1 + c \alpha_k \textrm{ for all } j \in \Iu.
    \end{align*}
    Let $v_1, \ldots, v_d$ be a basis of eigenvectors of $\nabla^2 f(0)$ associated to $h_1, \ldots, h_d$.
    These are also eigenvectors for all $T_k$ hence we may define the center-stable and unstable spaces as
    \begin{align*}
        \Ecs = \spann \{ v_i : i \in \Ics \}
        && \textrm{ and } && 
        \Eu = \spann \{ v_j : j \in \Iu \}.
    \end{align*}
    By construction, it holds that $T_k(\Ecs) \subseteq \Ecs$ and $T_k(\Eu) \subseteq \Eu$ for all $k$, so the invariance assumption~\aref{assumption:invariance} is satisfied with the same splitting for each pair $(g_k, T_k)$.
    
    \paragraph{Pseudo-hyperbolicity of the linear part.}
    Recall that the notion of NPH depends on the choice of norms on $\Ecs$ and $\Eu$.
    Here, we equip both $\Ecs$ and $\Eu$ with the (Euclidean) 2-norm, and use the max-norm from Eq.~\eqref{eq:Rd-norm} on $\reals^d$.
    
    Let us find constants $1 \leq \lambda_k < \mu_k$ such that $\| T_k y \| \leq \lambda_k \| y \|$ and $\| T_k z \| \geq \mu_k \| z \|$ for all $y \in \Ecs, z \in \Eu$.
    For all $y \in \Ecs$, we have
    \begin{align*}
        \| T_k y \| \leq 
        \Bigg( \sup_{\substack{x \in \Ecs \\ x \neq 0}} \frac{\| T_k x \|}{\| x \|} \Bigg) \; \| y \| = 
        \Bigg( \max_{i \in \Ics} \left| 1 - \alpha_k h_i \right|  \Bigg) \; \| y \| \leq \| y \|,
    \end{align*}
    hence we can choose $\lambda_k = 1$ for all $k \geq K$.
    (If $\Ics$ is empty, this remains a valid choice.)

    Similarly, for any $z \in \Eu$, we have that
    \begin{align*}
        \| T_k z \| \geq 
        \Bigg( \inf_{\substack{x \in \Eu \\ x \neq 0}} \frac{\| T_k x \|}{\| x \|} \Bigg) \; \| z \| = 
        \Bigg( \min_{j \in \Iu} \left| 1 - \alpha_k h_j \right| \Bigg) \; \| z \| \geq (1 + c \alpha_k) \| z \|,
    \end{align*}
    hence we can choose $\mu_k = 1 + c \alpha_k$ for all $k \geq K$.
    Observe $\mu_k > 1 = \lambda_k$, confirming~\aref{assumption:pseudo-hyperbolic}.

    \paragraph{Choosing $\varepsilon_k$ and concluding joint pseudo-hyperbolicity.}

    Set $\varepsilon_k = (c \alpha_k)/5$, so that (in particular) $\varepsilon_k = (\mu_k - \lambda_k)/5 < (\mu_k - \lambda_k)/4$.
    From Lemma~\ref{lemma:local-lipschitz-control-GD} we also get $r > 0$ such that
    \begin{align*}
        \Lip\Big( (g_k - T_k)\vert_{B_r(0)} \Big) \leq \frac{\varepsilon_k}{4}
        \quad\quad \textrm{ for all } k \geq K.
    \end{align*}
    This confirms \aref{assumption:small-lipschitz}.
    Combined with the above, it follows that
    \begin{align*}
        (g_k, T_k) \in \mathrm{PH}\Big(\mu_k, \lambda_k, \frac{\varepsilon_k}{4}; B_r(0), \Ecs \oplus \Eu \Big)
        \quad\quad \textrm{ for all } k \geq K.
    \end{align*}

    \paragraph{Verifying the non-summability condition.}

    The numbers $\frac{\varepsilon_k}{\mu_k - 2\varepsilon_k} = \frac{\frac{1}{5} c\alpha_k}{1 + \frac{3}{5} c\alpha_k}$ are non-summable by Lemma~\ref{lemma:summability}, owing to the assumption that the $\alpha_k$ themselves are non-summable. 

    All of the above, combined, confirms that $x^\ast$ (which was an arbitrary strict saddle point of $f$) is an NPH unstable fixed point.
\end{proof}

From Lemma~\ref{lemma:csmt-sequence-of-gradient-maps} and the global avoidance result in Theorem~\ref{thm:global-avoidance-nonautonomous}, we readily obtain the following saddle avoidance result:

\begin{proposition} \label{proposition:GD-admissible-corollary}
    Let $f \colon \reals^d \to \reals$ be $C^2$ and let $( \alpha_k )_{k \geq 0}$ be an admissible sequence of step sizes (Definition~\ref{def:admissiblesequence}) with $\sum_{k = 0}^\infty \alpha_k = \infty$.
    Assume $g_k(x) = x - \alpha_k \nabla f(x)$ satisfies the Luzin $N^{-1}$ property for all $k \geq 0$.
    Then, \eqref{eq:GD} avoids the strict saddle points of $f$.
\end{proposition}

Reconsider Theorem~\ref{thm:gd-vanishing-theorem}.
From Remark~\ref{remark:admissible-steps}, we see that all three step-size schedules are admissible.
They also satisfy $\sum_k \alpha_k = \infty$ owing to $\alpha_0 > 0$ and $\gamma \in (0, 1]$.
(For cosine decay, this is because the cosine term is periodic and never equal to $-1$.)
%
%
%
Thus, to establish the theorem, it remains to secure the Luzin $N^{-1}$ property: we do so next.

\subsection{The Luzin $N^{-1}$ property holds generically} \label{sec:luzin-gd}

To apply Proposition~\ref{proposition:GD-admissible-corollary}, we need each iteration map $g_k$ of \eqref{eq:GD} to satisfy the Luzin $N^{-1}$ property (Definition~\ref{def:luzin}).
This requirement is mild: the set of step sizes $\alpha > 0$ such that the map $g_\alpha(x) := x - \alpha \nabla f(x)$ lacks the Luzin $N^{-1}$ property has Lebesgue measure zero, see \citep[Lem.~3.7]{musat2025saddlelinesearch}.
Consequently, it is reasonable to expect that many step-size schedules $(\alpha_k)$ yield sequences of maps $g_k$ for \eqref{eq:GD} which \emph{all} have the Luzin $N^{-1}$ property.

We formalize one such setting.
When each step size $\alpha_k$ depends regularly on an initial parameter $\alpha_0$, we show that all corresponding iteration maps satisfy the Luzin $N^{-1}$ property for almost all $\alpha_0 > 0$.
(See also Remark~\ref{remark:gd-restricted-step-always-luzin} for an alternative based on a Lipschitz gradient assumption.)

\begin{lemma} \label{lemma:all-maps-have-luzin}
    Let $f \colon \reals^d \to \reals$ be $C^2$.
    For some $\alpha_0 > 0$, generate $(\alpha_k)_{k \geq 0}$ as $\alpha_k = h_k(\alpha_0)$, where each $h_k \colon \positive \to \positive$ 
    is $C^1$ and satisfies $h_k'(\alpha) \neq 0$ for almost all $\alpha > 0$.
    Then, for almost all $\alpha_0 > 0$, every iteration map $g_k$ of \eqref{eq:GD} satisfies the Luzin $N^{-1}$ property.  
\end{lemma}
\begin{proof}
    The set $A \coloneqq \{ \alpha > 0 \mid x \mapsto x - \alpha \nabla f(x) \textrm{ does not satisfy Luzin } N^{-1} \} $ has measure zero \citep[Lem.~3.7]{musat2025saddlelinesearch}.
    Since we assumed that $h_k' \neq 0$ almost everywhere, Lemma~\ref{lemma:luzin-verify} implies that $h_k$ itself satisfies the Luzin $N^{-1}$ property and so the set 
    \begin{align*}
        h_k^{-1}(A) = \{ \alpha_0 > 0 \mid g_k \textrm{ does not satisfy Luzin } N^{-1} \}
    \end{align*} 
    also has measure zero.
    Therefore, the set of all bad initial step sizes $\alpha_0$ is
    \begin{align*}
        \left\{ \alpha_0 \mid \exists k \geq 0 \textrm{ such that } g_k \textrm{ does not satisfy Luzin } N^{-1} \right\} = 
        \bigcup_{k \geq 0} h_k^{-1}(A).
    \end{align*}
    This has measure zero as a countable union of measure zero sets, as claimed.
\end{proof}

Theorem~\ref{thm:gd-vanishing-theorem} now follows from Proposition~\ref{proposition:GD-admissible-corollary} (and its subsequent comment) and from Lemma~\ref{lemma:all-maps-have-luzin} upon checking that it (indeed) applies to all three step-size schedules: they have $\alpha_k = h_k(\alpha_0)$ with $h_k'(\alpha) > 0$ for all $k \geq 0$ and all $\alpha > 0$.
(For cosine decay, this uses again that the cosine term is never equal to $-1$.)

\section{Extension to systems and algorithms on manifolds} \label{sec:extension-manifolds}

Riemannian gradient descent (RGD) for a cost function $f \colon \calM \to \reals$ on a Riemannian manifold $\calM$ corresponds to the non-autonomous dynamical system
\begin{align}
    g_k \colon \calM \to \calM, && g_k(x) = \Retr_x(-\alpha_k \grad f(x)),
    \tag{RGD}
    \label{eq:RGD}
\end{align}
where $\grad f$ is the (Riemannian) gradient of $f$ and $\Retr \colon \T\calM \to \calM$ is a \emph{retraction}, that is, a smooth map defined on the tangent bundle $\T\calM = \{ (x, v) : x \in \calM \textrm{ and } v \in \T_x\calM \}$ such that for all $(x, v)$ the curve $c(t) := \Retr_x(tv)$ satisfies $c(0) = x$ and $c'(0) = v$~\citep{AMS08}.
(This generalizes GD, where $\calM = \Rd$, $\grad f = \nabla f$ and $\Retr_x(v) = x+v$.)

To study saddle avoidance of RGD with vanishing step sizes, we first discuss how to extend Theorem~\ref{thm:global-avoidance-nonautonomous} to manifolds (Section~\ref{subsec:riemannian-reduction}), then we study RGD specifically, in the same way that we studied GD (Section~\ref{subsec:RGD}).

\subsection{Reduction to the Euclidean case} \label{subsec:riemannian-reduction}

As a first step, we must extend the notion of an NPH unstable fixed point $x^\ast$ to smooth manifolds $\calM$.
In the Euclidean setting, the small-Lipschitz assumption~\aref{assumption:small-lipschitz} is stated as
\begin{align*}
    \Lip((g-T)\vert_{U}) \leq \varepsilon
\end{align*}
for a map $g \colon \Rd \to \Rd$ and a linear map $T \colon \Rd \to \Rd$.
On a manifold $\calM$, however, $g$ maps $\calM$ to $\calM$ whereas $T$ is a linear map on the tangent space $\T_{x^\ast} \calM$.
Thus, the expression $g-T$ is not well defined.

One way to deal with this is to lift the iteration maps to the tangent space at the fixed point $x^\ast$.
In general, there is no diffeomorphism from $\calM$ to $\T_{x^\ast}\calM$ as a whole, but it is enough to have this locally.
Specifically, choose any pair of smooth maps
\begin{align}
    \psi \colon \T_{x^\ast}\calM \to \calM && \textrm{ and } && \hat\psi \colon \calM \to \T_{x^\ast}\calM
    \label{eq:psihatpsi}
\end{align}
together with a neighborhood $U$ of $x^\ast$ on $\calM$ and a neighborhood $V$ of the origin in $\T_{x^\ast}\calM$ such that (a) $\psi(0) = x^\ast$, and (b) the restrictions $\psi|_V \colon V \to U$ and $\hat\psi|_U \colon U \to V$ are each other's inverse.
(See Remark~\ref{rem:choiceofpsihatpsi} below for an example.)

Then, if $x^\ast$ is a fixed point for all iteration maps in the system $(g_k)_{k \geq 0}$ on $\calM$, we define the lifted maps
\begin{align} \label{eq:iteration-maps-lifts}
    \tilde{g}_k \coloneqq \hat\psi \circ g_k \circ \psi \colon \T_{x^\ast}\calM \to \T_{x^\ast}\calM.
\end{align}
The new sequence $(\tilde g_k)_{k \geq 0}$ defines a system on the vector space $\T_{x^\ast}\calM$, so that we can define NPH instability of $x^\ast$ by relying on Definition~\ref{def:unstable-fixed-point}, as follows.

\begin{definition} \label{def:NPHmanifolds}
    We call $x^\ast$ an \emph{NPH unstable fixed point} of $( g_k \colon \calM \to \calM )_{k \ge 0 }$ if there exist $\psi, \hat\psi$ as above such that $0 \in \T_{x^\ast}\calM$ is an NPH unstable fixed point of $(\tilde{g}_k \coloneqq \hat\psi \circ g_k \circ \psi)_{k \geq 0}$ in the sense of Definition~\ref{def:unstable-fixed-point}.
\end{definition}

We can now extend Theorem~\ref{thm:global-avoidance-nonautonomous} to manifolds.

\begin{theorem} \label{thm:global-avoidance-nonautonomous-riemannian}
    Let $\calM$ be a smooth manifold.
    If each map $g_k$ of the non-autonomous dynamical system $\big( g_k \colon \calM \to \calM \big)_{k \geq 0}$ has the Luzin $N^{-1}$ property, then the system avoids its set of NPH unstable fixed points (Definition~\ref{def:NPHmanifolds}).
\end{theorem}

Recall that the notions of ``avoidance'' (Definition~\ref{def:stable-set}) and ``Luzin $N^{-1}$'' (Definition~\ref{def:luzin}) require a concept of \emph{measure zero} set.
On manifolds, we use the standard chart-based notion~\citep[Ch.~6]{lee2012smoothmanifolds}:
a set $A \subset \calM$ has measure zero in $\calM$ if, for every smooth chart $(U,\varphi)$ of $\calM$, the coordinate image $\varphi(A \cap U) \subset \reals^d$ has (Lebesgue) measure zero.

\begin{proof}[Proof of Theorem~\ref{thm:global-avoidance-nonautonomous-riemannian}.]
    Let $x^\ast$ be an NPH unstable fixed point of $(g_k)_{k \geq 0}$.
    Summon $\psi, \hat\psi$ and $U, V$ as per Definition~\ref{def:NPHmanifolds}.
    Apply Theorem~\ref{theorem:csmt-general} to $(\tilde{g}_k)_{k \geq 0}$: this provides an integer $K \geq 0$, a neighborhood $\widetilde{B}$ of the origin in $\T_{x^\ast}\calM$ and a set $\Wcsloctilde$ of measure zero in $\T_{x^\ast} \calM$ with the following property:
    if $(v_k)_{k\geq 0}$ is a sequence such that, for some $\bar{k} \geq K$, we have $v_{k+1} = \tilde{g}_k(v_k)$ and $v_k \in \widetilde{B}$ for all $k \geq \bar{k}$, then $v_k$ is in $\Wcsloctilde$ for all $k \geq \bar{k}$.
    If need be, restrict both sets to $V \subseteq \T_{x^\ast}\calM$: the property remains valid.
    
    Push these sets to $\calM$ as $B(x^\ast) \coloneqq \psi(\widetilde{B})$ and $\Wcsloc(x^\ast) \coloneq \psi(\Wcsloctilde)$.
    Since $\psi|_V \colon V \to U$ is a diffeomorphism, we see that $B(x^\ast)$ is a neighborhood of $x^\ast$ and $\Wcsloc(x^\ast)$ has measure zero.

    Now consider a trajectory $(x_k)_{k \geq 0}$ with $x_{k+1} = g_k(x_k)$.
    Assume it eventually enters and never leaves $B(x^\ast)$, that is, there exists $\bar{k} \geq K$ such that $x_k \in B(x^\ast)$ for all $k \geq \bar{k}$.
    Since $B(x^\ast)$ is in $U$, we can lift those iterates to $\T_{x^\ast}\calM$ as $v_k \coloneqq \hat\psi(x_k)$ in such a way that $\psi(v_k) = x_k$.
    Thus, by design,
    \begin{align}
        v_{k+1} = \hat\psi(x_{k+1}) = \hat\psi(g_k(x_k)) = \hat\psi(g_k(\psi(v_k))) = \tilde g_k(v_k), && \forall k \geq \bar{k}.
        \label{eq:conjugatedynamicspsi}
    \end{align}
    In other words: $(v_k)_{k \geq \bar{k}}$ is (the tail-end of) a trajectory of the lifted system, and it is contained in $\tilde B$.
    Therefore, it is also contained in $\Wcsloctilde$, and (returning to the manifold) we find that $x_k$ is in $\Wcsloc(x^\ast)$ for all $k \geq \bar{k}$.
    
    From here, the Luzin $N^{-1}$ property allows to conclude as in the proof of Theorem~\ref{thm:global-avoidance-nonautonomous}.
\end{proof}

\begin{remark}[Choices of $\psi, \hat\psi$] \label{rem:choiceofpsihatpsi}
    There exists many possible choices of maps for~\eqref{eq:psihatpsi}.
    For example, on a complete Riemannian manifold, let $\psi \coloneqq \Exp_{x^\ast}$ be the (Riemannian) exponential map at $x^\ast$ (smooth, and defined globally since $\calM$ is complete) \citep[Prop.~5.19]{lee2018riemannian}.
    This has a smooth inverse $\Log_{x^\ast}$ when restricted to open balls of radius given by the injectivity radius $\inj(x^\ast) > 0$.
    A smooth transition function can be used to define globally a smooth $\hat\psi$ that matches $\Log_{x^\ast}$ on a ball of radius $\frac{1}{2}\inj(x^\ast)$ and then smoothly drops to $0 \in \T_{x^\ast}\calM$.
    This construction has the added benefit that $\D\psi(0)$ is identity on $\T_{x^\ast}\calM$, and likewise for $\D\hat\psi(x^\ast)$.
\end{remark}

\subsection{Riemannian Gradient Descent with vanishing step sizes} \label{subsec:RGD}

We show the following generalization of Theorem~\ref{thm:gd-vanishing-theorem}.
The restriction to real-analytic manifolds and retractions is mild: this holds for most applications, and the cost function $f$ only needs to be $C^2$.
This provision is used to secure the Luzin $N^{-1}$ property (see below).

\begin{theorem} \label{thm:RGDmain}
    Let $\calM$ be a real-analytic complete Riemannian manifold and let $f \colon \calM \to \reals$ be $C^2$.
    Assume that the retraction $\Retr \colon \T\calM \to \calM$ is such that for each $x \in \calM$, the map $v \mapsto \Retr_x(v)$ is real analytic.
    Consider \eqref{eq:RGD} with given $\alpha_0 > 0$ and one of the step-size schedules from Theorem~\ref{thm:gd-vanishing-theorem}.
    Then, for each fixed choice of the auxiliary parameters ($\gamma, T$) and for almost all $\alpha_0 > 0$,~\eqref{eq:RGD} avoids the strict saddle points of $f$.
\end{theorem}
\noindent This follows from Theorem~\ref{thm:global-avoidance-nonautonomous-riemannian}: we prove first that strict saddles of $f$ are NPH unstable fixed points of~\eqref{eq:RGD} (Lemma~\ref{lemma:rgd-saddle-avoidance}).
Then, leveraging the real-analyticity assumptions, we confirm that all~\eqref{eq:RGD} maps have the Luzin $N^{-1}$ property because each step size $\alpha_k$ depends regularly on the initial step size (Lemma~\ref{lemma:all-maps-have-luzin-riemannian}). 

\begin{lemma} \label{lemma:rgd-saddle-avoidance}
    Let $\calM$ be a complete Riemannian manifold and let $f \colon \calM \to \reals$ be $C^2$.
    If $( \alpha_k )_{k \geq 0}$ is admissible (Definition~\ref{def:admissiblesequence}) and bounded, and $\sum_{k \geq 0} \alpha_k = \infty$,
    then every strict saddle point of $f$ is an NPH unstable fixed point for \eqref{eq:RGD} (Definition~\ref{def:NPHmanifolds}).
\end{lemma}

\begin{proof}
Let $x^\ast$ be a strict saddle point of $f$.
Select maps $\psi, \hat\psi$ and neighborhoods $U, V$ conforming to~\eqref{eq:psihatpsi} as instructed in Remark~\ref{rem:choiceofpsihatpsi}.
Lift the iteration maps as $\tilde g_k \coloneqq \hat\psi \circ g_k \circ \psi$.

To show that $0 \in \T_{x^\ast}\calM$ is NPH unstable for $(\tilde{g}_k)_{k \geq 0}$ in the sense of Definition~\ref{def:unstable-fixed-point}, as usual, we start by defining the linear maps
\begin{align*}
    T_k \coloneqq 
    \D \tilde{g}_k(0)
    & = \D\hat\psi(g_k(\psi(0))) \circ \D g_k(\psi(0)) \circ \D\psi(0) \\
    & = \D\hat\psi(x^\ast) \circ \D g_k(x^\ast) \circ \D\psi(0) = \D g_k(x^\ast),
\end{align*}
where we used $\psi(0) = x^\ast$ and $g_k(x^\ast) = x^\ast$, as well as the fact that $\D\psi(0)$ and $\D\hat\psi(x^\ast)$ are identity (as noted in Remark~\ref{rem:choiceofpsihatpsi}).
It follows that \citep[Lem.~3.8]{musat2025saddlelinesearch}:
\begin{align} \label{eq:T_k_riemannian}
    T_k = 
    \D g_k(x^\ast) = 
    I - \alpha_k \Hess f(x^\ast),
\end{align}
where $\Hess f$ denotes the (Riemannian) Hessian of $f$.

We now check all conditions in Definition~\ref{def:unstable-fixed-point}, starting with the pairs $(\tilde g_k, T_k)$.

First, $\tilde{g}_k(0) = \hat\psi(g_k(\psi(0))) = 0$ for all $k$, so~\aref{assumption:fixed-origin} is satisfied.
Next, $T_k$~\eqref{eq:T_k_riemannian} is diagonalizable in an eigenbasis of $\Hess f(x^\ast)$.
Hence, \aref{assumption:invariance} and \aref{assumption:pseudo-hyperbolic} follow exactly as in the Euclidean case (Lemma~\ref{lemma:csmt-sequence-of-gradient-maps}).
Some details follow.

The admissibility of $(\alpha_k)$ yields a fixed splitting $\T_{x^\ast} \calM = \Ecs \oplus \Eu$ such that $T_k(\Ecs)\subseteq \Ecs$ and $T_k(\Eu)\subseteq \Eu$ for all $k$.
Recall that NPH depends on a choice of norms; here we use the norms induced by the Riemannian metric at $x^\ast$ when restricted to $\Ecs$ and $\Eu$, and equip $\T_{x^\ast} \calM = \Ecs \oplus \Eu$ with the associated max-norm as in Eq.~\eqref{eq:Rd-norm}.
Then, as in the Euclidean case, for all large enough $k$,
$\|T_k y\| \leq \|y\|$ on $\Ecs$ and $\| T_k z\| \geq (1+c\alpha_k) \|z\|$ on $\Eu$.
In particular, we may take $\lambda_k \coloneqq 1$ and $\mu_k \coloneqq 1+c\alpha_k$ to satisfy~\aref{assumption:pseudo-hyperbolic}.

Choosing $\varepsilon_k \coloneqq (c \alpha_k)/5$ ensures $\varepsilon_k < (\mu_k - \lambda_k)/4$.
Since the choice of constants $\mu_k, \lambda_k, \varepsilon_k$ is the same as in the Euclidean case, the non-summability condition for $\frac{\varepsilon_k}{\mu_k - 2\varepsilon_k}$ follows identically.

It remains to verify the $(\varepsilon_k / 4)$-Lipschitz control required by~\aref{assumption:small-lipschitz}.
To this end, it suffices to show that, for all $c > 0$, there exist an index $k_0 \geq 0$ and some radius $r > 0$ (small enough so that $B_r(0)$ is included in $V$) such that 
\begin{align} \label{eq:rgd-small-lipschitz}
    \Lip \Big( (\tilde{g}_k - \D \tilde{g}_k(0))\vert_{B_{r}(0)} \Big) \leq c \alpha_k
    \qquad \textrm{for all } k \geq k_0.
\end{align}
In the Euclidean case, we did so with Lemma~\ref{lemma:local-lipschitz-control-GD}.
However, that lemma crucially relies on the explicit identity $\D g_k(x) = I - \alpha_k \nabla^2 f(x)$ holding for all points $x$.
In the Riemannian setting, this is only satisfied at critical points.
Nevertheless, we can define $g_\alpha(x) = \Retr_x(-\alpha \grad f(x))$ and $\tilde{g}_\alpha = \hat\psi \circ g_\alpha \circ \psi$, then
\begin{equation*}
    h \colon \reals \times \T_{x^\ast} \calM \to \textrm{Lin}(\T_{x^\ast} \calM,  \T_{x^\ast} \calM) \colon (\alpha, v) \mapsto h(\alpha, v) = \D \tilde{g}_{\alpha}(v) - \D \tilde{g}_{\alpha}(0).
\end{equation*}
Observe that
$h(\alpha, 0) = 0$ for all $\alpha$ (trivially), and
$h(0, v) = 0$ for all $v \in V$ (because $g_0$ is identity and $\tilde{g}_0 = \hat\psi \circ \psi$ is identity on $V$).
Also, $h$ is smooth in $\alpha$ and jointly continuous in $(\alpha, v)$.
Then, since the step sizes are bounded (note this is the only place where the boundedness assumption is used), the desired Lipschitz control in Eq.~\eqref{eq:rgd-small-lipschitz} follows from Lemma~\ref{lemma:uniform-alpha-control}.
\end{proof}

To prove Theorem~\ref{thm:RGDmain}, it only remains to verify the Luzin $N^{-1}$ property for the iteration maps $g_k$ of~\eqref{eq:RGD}.
The following lemma provides what is needed:

\begin{lemma} \label{lemma:all-maps-have-luzin-riemannian}
    Assume $\calM$ is a real-analytic manifold with a Riemannian metric and a retraction $\Retr \colon \T\calM \to \calM$ such that, for each $x \in \calM$, the map $v \mapsto \Retr_x(v)$ is real analytic.
    Let $f \colon \calM \to \reals$ be $C^2$.
    Then, under the same step size generation assumptions as in Lemma~\ref{lemma:all-maps-have-luzin}, for almost all $\alpha_0 > 0$, every iteration map $g_k$ of~\eqref{eq:RGD} satisfies the Luzin $N^{-1}$ property.
\end{lemma}
\noindent The proof is identical to that of Lemma~\ref{lemma:all-maps-have-luzin} because the set 
\begin{align*}
    A \coloneqq \{ \alpha > 0 \mid x \mapsto \Retr_x(-\alpha \grad f(x)) \textrm{ does not satisfy Luzin } N^{-1} \}
\end{align*}
has measure zero \citep[Thm.~3.11]{musat2025saddlelinesearch}, owing to real analyticity.

\section{Proximal point method with non-constant step sizes} \label{section:pp}

For a different application of Theorem~\ref{thm:global-avoidance-nonautonomous}, consider a $C^2$ cost function $f \colon \reals^d \to \reals$ with $L$-Lipschitz continuous gradient.
The proximal point method for minimizing $f$ initialized at $x_0$ iterates $x_{k+1} = g_k(x_k)$ using the maps
\begin{align}
    g_k(x) = \argmin_{z \in \reals^d} \! \left[ f(z) + \frac{1}{2 \alpha_k} \| z - x \|_2^2 \right]
    \tag{PP}
    \label{eq:pp-update}
\end{align}
with step sizes $\alpha_k \in (0, 1/L)$.
These form a non-autonomous dynamical system.
We show that this method, too, avoids strict saddles.
The proof is a combination of the three lemmas below.
\begin{theorem} \label{thm:pp-avoidance}
    Let $f \colon \reals^d \to \reals$ be $C^2$ with $L$-Lipschitz continuous gradient and let $(\alpha_k)_{k \geq 0}$ be a sequence of step sizes in $(0, 1/L)$.
    If $\alpha_{\max} \coloneqq \sup_{k \geq 0} \alpha_k < 1/L$ and $\sum_{k=0}^\infty \alpha_k = \infty$, then \eqref{eq:pp-update} avoids the strict saddle points of $f$.
\end{theorem}

\begin{remark}
    Unlike Theorems~\ref{thm:gd-vanishing-theorem} and~\ref{thm:RGDmain}, this result does not require an ``almost all'' qualification on the choice of the initial step size.
    The reason is that we restrict the steps to satisfy $\alpha_k \in (0, 1/L)$, which ensures that all maps have the Luzin $N^{-1}$ property.
    Under the same assumptions, an identical conclusion holds for GD and RGD; see Remark~\ref{remark:gd-restricted-step-always-luzin}.
\end{remark}

Let us first argue that for $\alpha_k \in (0, 1/L)$ the maps $g_k$ are well defined.
Since $\nabla f$ is $L$-Lipschitz and $f$ is $C^2$, we have $\nabla^2 f(z) \succeq - L I$ for all $z$, hence 
$ \nabla^2 f(z) + \frac{1}{\alpha_k} I \succeq (\frac{1}{\alpha_k} - L) I \succ 0 $.
Thus, the objective to be minimized in the~\eqref{eq:pp-update} iteration is strongly convex: it has a unique minimizer.

As in previous sections, we deduce Theorem~\ref{thm:pp-avoidance} from Theorem~\ref{thm:global-avoidance-nonautonomous}.
To do so, we show that the strict saddles of $f$ are NPH unstable fixed points of the system $(g_{k})_{k \geq 0}$.
As a first step, in the next lemma we show that the~\eqref{eq:pp-update} maps are $C^1$ diffeomorphisms (this is standard).
In particular, they satisfy the Luzin $N^{-1}$ property by Lemma~\ref{lemma:luzin-verify}.

\begin{lemma} \label{lemma:pp-inverse-gd-minus-f}
    Let $f \colon \reals^d \to \reals$ be $C^2$ with $L$-Lipschitz continuous gradient.
    For all $\alpha \in (0, 1/L)$, the map $g_{\alpha}(x) = \argmin_{z \in \reals^d}\!\big[ f(z) + \frac{1}{2\alpha} \| z - x \|_2^2 \big]$ is a $C^1$ diffeomorphism from $\Rd$ to $\Rd$, with inverse $u_\alpha(x) = x + \alpha \nabla f(x)$.
    Its differential at $x \in \reals^d$ is
    \begin{align}
        \D g_\alpha(x) = 
        \big( I + \alpha \nabla^2 f(g_\alpha(x)) \big)^{-1}.
        \label{eq:D_g_alpha_pp}
    \end{align}
\end{lemma}
\begin{proof}
    By the optimality condition for the minimization problem within $g_\alpha$, we have that
    \begin{align}
        x = \alpha \nabla f(g_{\alpha}(x)) + g_{\alpha}(x) = u_{\alpha}(g_{\alpha}(x)) 
        && \textrm{ for all } x \in \reals^d.
        \label{eq:pp-optim}
    \end{align}
    In particular, $u_\alpha$ is surjective.
    Also, $u_{\alpha}$ is injective owing to $\alpha < 1/L$ since, for all $x, y \in \reals^d$,
    \begin{align*}
        \| u_{\alpha}(x) - u_{\alpha}(y) \|_2 = 
        \| (x - y) + \alpha (\nabla f(x) - \nabla f(y)) \|_2 \geq
        (1 - \alpha L) \| x - y \|_2.
    \end{align*}
    Thus, $u_\alpha$ is bijective, and its inverse is $g_\alpha$ owing to Eq.~\eqref{eq:pp-optim}.

    Since $\D u_{\alpha}(x) = I + \alpha \nabla^2 f(x)$ and $\nabla^2 f(x) \succeq -L I$, it follows (again from $\alpha < 1/L$) that $\D u_{\alpha}(x)$ is invertible.
    By the inverse function theorem, for each $x$ there exists a neighborhood on which $u_\alpha$ has a unique inverse, and that inverse is $C^1$.
    This local inverse must coincide with $g_\alpha$, hence $g_\alpha$ is $C^1$ and, as claimed, 
    $ \D g_\alpha(x) = \big( \D u_\alpha(g_\alpha(x)) \big)^{-1} = \big( I + \alpha \nabla^2 f(g_\alpha(x)) \big)^{-1} $.
\end{proof}

The following is an analogue of Lemma~\ref{lemma:local-lipschitz-control-GD} for the proximal point iteration maps.

\begin{lemma} \label{lemma:local-lipschitz-control-PP}
    Let $f \colon \reals^d \to \reals$ be $C^2$ with $L$-Lipschitz gradient.
    Assume $\nabla f(0) = 0$.
    Then, for all $c > 0$ and $\alpha_{\max} < 1/L$, there exists $r > 0$ such that all~\eqref{eq:pp-update} maps $g_k$ with $0 < \alpha_k \leq \alpha_{\max}$ satisfy 
    \begin{equation} \label{eq:lipschitz-bound-pp}
        \Lip \big( (g_k - \D g_k(0))\vert_{B_r(0)} \big) \leq c \alpha_k.
    \end{equation}
\end{lemma}
\begin{proof}
    Fix $\alpha_{\max} < 1/L$.
    For $\alpha \in (0, \alpha_{\max}]$, consider the maps $g_\alpha$ and $u_\alpha$ as in Lemma~\ref{lemma:pp-inverse-gd-minus-f}.
    From that lemma, $g_{\alpha}$ is a diffeomorphism with inverse $u_{\alpha}$, and $\D g_{\alpha}(x)$ satisfies Eq.~\eqref{eq:D_g_alpha_pp}.
    This allows us to bound its operator norm (induced by the Euclidean norm $\| \cdot \|_2$) as
    \begin{align} \label{eq:pp-Dg-bound}
        \| \D g_\alpha(x) \| = 
        \frac{1}{\sigma_{\min} (I + \alpha \nabla^2 f(g_\alpha(x)))} \leq 
        \frac{1}{1 - \alpha L} \leq 
        \frac{1}{1 - \alpha_{\max} L} \eqqcolon \rho.
    \end{align}    
    Define the $C^1$ function $\xi(z) = \nabla f(z) - \nabla^2 f(0)[z]$.
    Plugging this in Eq.~\eqref{eq:pp-optim}, we obtain
    \begin{align*}
        x = 
        \alpha \Big( \nabla^2 f(0) [g_{\alpha}(x)] + \xi (g_{\alpha}(x)) \Big) + g_{\alpha}(x) =
        \Big(I + \alpha \nabla^2 f(0)\Big) g_{\alpha}(x) + \alpha \xi (g_{\alpha}(x)).
    \end{align*}
    Notice $u_\alpha(0) = 0$ hence $g_\alpha(0) = 0$.
    Thus, we get from Eq.~\eqref{eq:D_g_alpha_pp} that $\D g_\alpha(0) = \big(I + \alpha \nabla^2 f(0)\big)^{-1}$.
    Apply $\D g_\alpha(0)$ to both sides of the previous equation to obtain
    \begin{align*}
        g_\alpha(x) -  \D g_{\alpha}(0)[x] = - \alpha \D g_{\alpha}(0) [\xi (g_{\alpha}(x))].
    \end{align*}
    Therefore, for all $r > 0$, 
    \begin{align}
        \Lip \big( (g_{\alpha} - \D g_{\alpha}(0) )\vert_{B_r(0)} \big) \leq
        \alpha \; \| \D g_{\alpha} (0) \| \; \Lip \big( (\xi \circ g_{\alpha})\vert_{B_r(0)} \big) \leq 
        \alpha \rho \; \Lip((\xi \circ g_\alpha)\vert_{B_r(0)}).
        \label{eq:lipschitz-decomp}
    \end{align}
    By continuity of $\D \xi(z) = \nabla^2 f(z) - \nabla^2 f(0)$, for any $\delta > 0$, we can pick $R > 0$ such that for all $z$ with $\| z \| \leq R$, it holds that $\| \D \xi(z) \| \leq \delta / \rho$.
    Since $g_\alpha$ is $\rho$-Lipschitz (by Eq.~\eqref{eq:pp-Dg-bound}) and $g_\alpha(0) = 0$, we find $ \| g_\alpha(x) \| \leq \rho \| x \| $.
    Let $r = R / \rho$.
    If $ \| x \| \leq r $, then $\| g_\alpha(x) \| \leq R$, so
    \begin{align*}
        \Lip \big( (\xi \circ g_\alpha)\vert_{B_r(0)} \big) \leq
        \max_{\| x \| \leq r} \| \D \xi(g_\alpha(x)) \circ \D g_\alpha(x) \| \leq
        \max_{\| z \| \leq R} \| \D \xi(z) \| \; \max_{\| x \| \leq r} \| \D g_\alpha(x) \| \leq
        \delta.
    \end{align*}
    Going back to Eq.~\eqref{eq:lipschitz-decomp}, this yields $\Lip \big( (g_{\alpha} - \D g_{\alpha}(0) )\vert_{B_r(0)} \big) \leq \alpha \rho \delta $.
    Accordingly, given $c > 0$, we can pick $\delta = c / \rho$ and the associated $r > 0$ to conclude.
\end{proof}

The next lemma confirms that strict saddles of $f$ are NPH unstable fixed points of~\eqref{eq:pp-update}.

\begin{lemma} \label{lemma:saddles-unstable-pp}
    Let $f \colon \reals^d \to \reals$ be $C^2$ with $L$-Lipschitz continuous gradient and let $(\alpha_k)_{k \geq 0}$ be a sequence of step sizes in $(0, 1/L)$.
    If $\alpha_{\max} \coloneqq \sup_{k \geq 0} \alpha_k < 1/L$ and $\sum_{k=0}^\infty \alpha_k = \infty$,
    then the strict saddle points of $f$ are NPH unstable fixed point of~\eqref{eq:pp-update}.
\end{lemma}
\begin{proof}
Let $x^\ast$ be a strict saddle point of $f$.
Assume without loss of generality that $x^\ast = 0$.
Let us check the conditions in Definition~\ref{def:unstable-fixed-point} for the sequence of maps $(g_k)_{k \geq 0}$ and the associated linear maps $T_k = \D g_k(0)$.

Recall $g_k(0) = 0$ for all $k$, hence \aref{assumption:fixed-origin} holds.
Denote the eigenvalues of $\nabla^2 f(0)$ by $h_1 \geq \cdots \geq h_d$, with associated eigenvectors $v_1, \ldots, v_d$.
From Eq.~\eqref{eq:D_g_alpha_pp}, we know $\D g_k(0) = (I + \alpha_k \nabla^2 f(0))^{-1}$ has eigenvalues $1/(1+\alpha_k h_1) \leq \cdots \leq 1/(1+\alpha_k h_d)$ with the same eigenvectors.

Let $s$ be the index such that $h_1, \ldots, h_s \in [0, L]$ and $h_{s+1}, \ldots, h_d \in [-L, 0)$.
Accordingly, $1/(1+\alpha_k h_i)$ is in $(1/2, 1]$ for $i \leq s$ and it is strictly greater than $1$ for $i > s$.
Let $\Ecs = \mathrm{span}\{ v_1, \ldots, v_s \}$ and $\Eu = \mathrm{span}\{ v_{s+1}, \ldots, v_d \}$.
This guarantees invariance as per \aref{assumption:invariance}.
Endow $\Ecs$ and $\Eu$ with the Euclidean norm $\| \cdot \|_2$.
We choose constants $\lambda_k \coloneqq 1$ and $\mu_k \coloneqq 1 / (1 + \alpha_k h_{s+1}) > 1$ for all $k \geq 0$, so~\aref{assumption:pseudo-hyperbolic} is satisfied.

Next, observe that for all $k \geq 0$ we have
\begin{equation*}
    \mu_k - \lambda_k = \frac{- \alpha_k h_{s+1}}{1 + \alpha_k h_{s+1}} > - \alpha_k h_{s+1} > 0.
\end{equation*}
Let $\varepsilon_k \coloneqq \frac{-h_{s+1}}{5} \alpha_k > 0$.
This ensures $\varepsilon_k < (\mu_k - \lambda_k) / 4$ as required by Definition~\ref{def:unstable-fixed-point}.
To complete \aref{assumption:small-lipschitz}, invoke Lemma~\ref{lemma:local-lipschitz-control-PP}: there exists $r > 0$ such that $\Lip( (g_k - T_k)\vert_{B_r(0)}) \leq \varepsilon_k / 4$
for all $k \geq 0$.

It only remains to check the non-summability of $\varepsilon_k/(\mu_k - 2\varepsilon_k)$.
To this end, let $\beta \coloneqq -h_{s+1} \in (0, L]$ and $t_k \coloneqq \beta \alpha_k \in (0, t_{\max}]$, where $t_{\max} \coloneqq \beta \alpha_{\max} \in (0, 1)$.
Since 
\begin{align*}
    \frac{\varepsilon_k}{\mu_k - 2 \varepsilon_k} = 
    \frac{t_k (1 - t_k)}{5 - 2 t_k (1 - t_k)} \geq
    \frac{t_k (1 - t_{\max})}{5} = 
    \frac{\beta (1 - t_{\max})}{5} \alpha_k,
\end{align*} 
the conclusion follows immediately from the non-summability of $\alpha_k$.
\end{proof}

\section{Perspectives}

We conclude by listing a few questions.

\begin{itemize}
    \item \textbf{Broader algorithmic scope.}
    Our approach views algorithms that apply a different update at each time step as non-autonomous dynamical systems, and it relies on a shared form of pseudo-hyperbolicity near saddle points.
    For which other algorithms could we apply the same technique to show avoidance of saddle points?

    \item \textbf{Non-summability.}
    Theorem~\ref{theorem:local-csmt-generic} requires the series $\sum_k \varepsilon_k / (\mu_k - 2\varepsilon_k)$ to be non-summable.
    This condition is more permissive than the uniform contraction/expansion requirements imposed in existing non-autonomous versions of the CSMT such as the one by \citet[Thm.~6.2.8]{katok1995introduction}: we needed this additional leeway to cover vanishing step-size regimes.
    For \eqref{eq:GD}, \eqref{eq:RGD} and \eqref{eq:pp-update}, we showed that the condition is satisfied as soon as the step sizes themselves are non-summable, which is rather natural for optimization methods.
    Still, it is natural to ask: could we construct a center-stable set with weaker (or no) non-summability assumption?
\end{itemize}

\section{Funding and Conflicts of interests}

This work was supported by the Swiss State Secretariat for Education, Research and Innovation (SERI) under contract number MB22.00027.
The authors declare no conflicts of interest.

\appendix

\section{Technical lemmas}

\begin{lemma}[Uniform contraction principle] \label{lemma:uniform-contraction}
    Let $Y, Z$ be two non-empty metric spaces and assume $Z$ is complete.
    Let $h \colon Y \times Z \to Z$ satisfy the following for some $0 \leq \ell < 1$ and $L \geq 0$ 
    uniformly for all $y_1, y_2 \in Y$ and $z_1, z_2 \in Z$:
    \begin{align*}
        \dist( h(y_1, z_1), h(y_1, z_2) ) &\leq \ell \; \dist(z_1, z_2), \textrm{ and } \\ 
        \dist( h(y_1, z_1), h(y_2, z_1) ) &\leq L \; \dist(y_1, y_2).
    \end{align*}
    Define $\xi \colon Y \to Z$ such that $\xi(y)$ is the unique fixed point of the contraction 
    $h(y, \cdot) \colon Z \to Z$.
    Then, $\xi$ is well defined and it is $L / (1 - \ell)$-Lipschitz continuous.
\end{lemma}
\begin{proof}
    By the contraction mapping principle, $\xi$ is well defined and satisfies $h(y, \xi(y)) = \xi(y)$ for all $y \in Y$.
    To show that it is Lipschitz continuous, observe that for any $y_1, y_2 \in Y$ we have 
    \begin{align*}
        \dist(\xi(y_1), \xi(y_2)) &= 
        \dist(h(y_1, \xi(y_1)), h(y_2, \xi(y_2))) \\ &\leq 
        \dist(h(y_1, \xi(y_1)), h(y_1, \xi(y_2))) + \dist(h(y_1, \xi(y_2)), h(y_2, \xi(y_2))) \\ &\leq 
        \ell \dist(\xi(y_1), \xi(y_2)) + L \dist(y_1, y_2).
    \end{align*}
    We obtain the conclusion by rearranging.
\end{proof}

\begin{lemma} [Local to global Lipschitz control] \label{lemma:local-to-global-small-lipschitz}
    Fix any norm $\| \cdot \|$ on $\reals^d$.
    Let $\big( g_k \colon \reals^d \to \reals^d \big)_{k \geq 0}$ be a sequence of maps with $g_k(0) = 0$ and $( T_k \colon \reals^d \to \reals^d )_{k \geq 0}$ be a sequence of linear maps.
    Suppose there exist $r > 0$ and positive $(\varepsilon_k)_{k \geq 0}$ such that 
    $\Lip\big( (g_k - T_k) \vert_{B_r(0)} \big) \leq \varepsilon_k / 4$ for all $k \geq 0$.
    Then, there exists a sequence of maps $( \tilde{g}_k \colon \reals^d \to \reals^d )_{k \geq 0}$ such that 
    \begin{align*}
        \tilde{g}_k = g_k \textrm{ on } B_{r/2}(0) 
        \quad\quad \textrm{ and } \quad\quad 
        \Lip \big( \tilde{g}_k - T_k \big) \leq \varepsilon_k.
    \end{align*}
\end{lemma}
\begin{proof}
    For each $k \geq 0$, set $u_k = g_k - T_k$.
    Notice $u_k(0) = 0$.
    Define the bump function
    \begin{align*}
        q(x) \coloneqq 
        \min\!\Big\{ 1, \max\!\big\{ 0, 2 - \frac{2}{r} \| x \| \big\} \Big\} \in [0, 1],
    \end{align*}
    and observe that $q(x) = 1$ for all $x \in B_{r/2}(0)$ and $q(x) = 0$ for all $x \in \reals^d \setminus B_r(0)$.
    Moreover, $q$ is Lipschitz continuous with $\Lip(q) \leq 2/r$.
    For each $k \geq 0$, define $\tilde{g}_k \colon \reals^d \to \reals^d$ as
    \begin{align*}
        \tilde{g}_k(x) \coloneqq
        T_k \, x + q(x) u_k(x).
    \end{align*}
    Since $q \equiv 1$ on $B_{r/2}(0)$, this satisfies $\tilde{g}_k = g_k$ on $B_{r/2}(0)$.
    It remains to bound $\Lip(q u_k)$.

    If $x, y \in B_r(0)$, using $u_k(0) = 0$, $q(x) \in [0, 1]$, $\Lip(u_k\vert_{B_r(0)}) \leq \varepsilon_k / 4$ and $\Lip(q) \leq 2/r$, 
    \begin{align}
        \| u_k(x) q(x) - u_k(y) q(y) \|
        & \leq \|  u_k(x) \| \; \left| q(x) - q(y) \right| + \left| q(y) \right| \; \| u_k(x) - u_k(y) \| \label{eq:Lipuqintermediate} \\ 
        & \leq \Lip \big(u_k \vert_{B_r(0)} \big) \; \| x \| \; \Lip(q) \; \| x - y \| + \Lip \big(u_k \vert_{B_r(0)} \big) \; \| x - y \| \nonumber\\
        & \leq \frac{3}{4} \varepsilon_k \; \| x - y \|. \nonumber
    \end{align}
    Otherwise, if $x \in B_r(0)$ and $y \in \reals^d \setminus B_r(0)$, resume from~\eqref{eq:Lipuqintermediate} and use $q(y) = 0$ to get
    \begin{align*}
        \| u_k(x) q(x) - u_k(y) q(y) \|
        & \leq \|  u_k(x) \| \; \left| q(x) - q(y) \right| \\
        & \leq \Lip \big(u_k \vert_{B_r(0)} \big) \; \| x \| \; \Lip(q) \; \| x - y \|
        \leq \frac{1}{2} \varepsilon_k \; \| x - y \|.
    \end{align*}
    Finally, if both $x$ and $y$ are outside $B_r(0)$, then $\| u_k(x) q(x) - u_k(y) q(y) \| = 0$.
    Therefore, we showed that $\Lip(\tilde{g}_k - T_k) \leq \frac{3}{4}\varepsilon_k$ for all $k \geq 0$.
\end{proof}

\begin{lemma} \label{lemma:summability}
    Let $( \alpha_k )_{k \geq 0}$ be a positive sequence with $\sum_{k \geq 0} \alpha_k = \infty$ and let $Q > 0$.
    Then, $\sum_{k \geq 0} \alpha_k / (1 + \alpha_k Q) = \infty$.
\end{lemma}
\begin{proof}
    We study two cases, depending on whether there are finitely or infinitely many indices $k$ such that $\alpha_k Q \leq 1$.
    Assume first that there is an infinite family $\mathcal{I}$ of such indices.
    Then
    \begin{align*}
        \sum_{k = 0}^\infty \frac{\alpha_k}{1 + \alpha_k Q} \geq 
        \sum_{k \in \mathcal{I}} \frac{\alpha_k}{1 + \alpha_k Q} \geq 
        \frac{1}{2} \sum_{k \in \mathcal{I}} \alpha_k = \infty.
    \end{align*}    
    In the second case, assume we have an infinite family $\mathcal{I}$ of indices $k$ such that $\alpha_k Q \geq 1$.
    Then
    \begin{align*}
        \sum_{k = 0}^\infty \frac{\alpha_k}{1 + \alpha_k Q} \geq 
        \sum_{k \in \mathcal{I}} \frac{\alpha_k}{1 + \alpha_k Q} \geq 
        \sum_{k \in \mathcal{I}} \frac{1}{2Q} = \infty,
    \end{align*}
    so indeed the sequence is non-summable in both cases.
\end{proof}

\begin{lemma} \label{lemma:uniform-alpha-control}
    Let $h \colon \reals \times \Rd \to \Rn$ be $C^1$ in its first argument.
    Assume
    \begin{itemize}
        \item For every $\alpha \in \reals$, $h(\alpha, 0) = 0$,
        \item There exists a neighborhood $V \subset \Rd$ of $0$ such that $h(0, v) = 0$ for all $v \in V$, and
        \item The partial derivative with respect to the first variable, $\partial_{\alpha} h$, is continuous.
    \end{itemize}
    Then, for all $\bar{\alpha} > 0$ and $c > 0$, there exists $r > 0$ such that
    \begin{align*}
        \sup_{\| v \| \leq r } \| h(\alpha, v) \| \leq c \alpha
        \quad \textrm{ for all } \alpha \in [0, \bar{\alpha}].
    \end{align*}
    (This holds for all norms on $\Rd$ and $\Rn$, only affecting the value of $r$.)
\end{lemma}
\begin{proof}
    Fix $\bar\alpha>0$ and $c>0$.
    Because $h(\alpha,0)=0$ for all $\alpha$, we have
    \[
    \partial_\alpha h(\alpha,0)=0
    \quad\text{for all }\alpha.
    \]
    This implies that the function $\omega \colon [0, \infty) \to \reals$ defined below satisfies $\omega(0) = 0$:
    \begin{align*}
        \omega(r) \coloneqq \sup_{\alpha \in [0, \bar\alpha], \|v\| \leq r} \|\partial_\alpha h(\alpha,v)\|.
    \end{align*}
    As in Lemma~\ref{lemma:local-lipschitz-control-GD},
    notice that $\omega$ is continuous owing to the Maximum Theorem and continuity of $(\alpha, v) \mapsto \partial_\alpha h(\alpha, v)$~\citep[p.~116]{berge1963topological}.
    Therefore, we can find $r > 0$ such that $\omega(r) \leq c$.
    If need be, reduce $r$ such that $B_r(0)$ is contained in $V$.
    For $\alpha \in [0, \bar{\alpha}]$ and $\|v\| \leq r$, using $h(0, v) = 0$, we find
    \[
    \|h(\alpha,v)\| = \left\| h(0, v) + \int_0^\alpha \partial_\alpha h(\tau,v)\,\mathrm{d}\tau \right\|
    \le \int_0^\alpha \|\partial_\alpha h(\tau,v)\|\,\mathrm{d}\tau
    \le \omega(r) \alpha \leq c\alpha.
    \]
    Taking the supremum over $\|v\|\le r$ completes the proof.
\end{proof}

\section{Proofs about the graph transform} \label{appendix:graph-transform-proofs}

This appendix holds proofs for some standard lemmas about the classical graph transform, as stated in Section~\ref{section:graph-transform-autonomous}.
They are extracted from the book by~\citet[Thm.~5.1, p56]{hirsch1977invariant}, where the authors attribute the core ideas to \citet{hadamard1901iteration}.
The lemmas and proofs as laid out below appeared earlier in a blog post by the authors.\footnote{\href{https://www.racetothebottom.xyz/posts/saddle-avoidance-general/}{racetothebottom.xyz/posts/saddle-avoidance-general/}}

Observe that $\pcs$ and $\pu$ (as defined in the introduction of Section~\ref{section:graph-transform-autonomous}) are $1$-Lipschitz with respect to the max-norm in Eq.~\eqref{eq:Rd-norm}.\footnote{$\| \pu(x) - \pu(x') \| = \| \pu(x - x')\| \leq \max(\| \pu(x - x')\|, \| \pcs(x - x') \|) = \| x - x' \|_{\reals^d} $ for any $x, x' \in \reals^d$.}
This property is used repeatedly below.

\subsection{Proof of Lemma~\ref{lemma:h_y_phi_lipschitz}} \label{sec:proof_lemma:h_y_phi_lipschitz}
\begin{proof} 
    Fix an arbitrary $y \in \Ecs$.
    For any $z, z' \in \Eu$, we have 
    \begin{align} \label{eq:h_y_varphi_contraction_bound}
        \mu \| \hyvarphi(z) - \hyvarphi(z') \| 
        & \leq \| \TEu \big( \hyvarphi(z) - \hyvarphi(z') \big) \| \notag \\
        & \leq \big\| \varphi(\gcs(y, z)) - (\gu - \Tu)(y, z) - \varphi(\gcs(y, z')) + (\gu - \Tu)(y, z') \big\| \notag \\
        & \leq \big\| \varphi(\gcs(y, z)) - \varphi(\gcs(y, z'))  \big\| + \big\|  (\gu - \Tu)(y, z)  - (\gu - \Tu)(y, z') \big\| \notag \\ 
        & \leq \Lip(\varphi) \, \| \gcs(y, z) - \gcs(y, z') \| + \Lip(\pu \circ (g - T) ) \, \| z - z' \|,
    \end{align}  
    where in the first step we used~\aref{assumption:pseudo-hyperbolic}, and in the second step we used 
    the definition of $\hyvarphi$ from Eq.~\eqref{eq:h-phi-y-definition}.
    For the first term, we split $g = T + (g - T)$
    \begin{align*}
        \| \gcs(y, z) - \gcs(y, z') \| 
        & \leq \| (\gcs-\Tcs)(y, z) - (\gcs-\Tcs)(y, z') \| + \| \Tcs(y, z) - \Tcs(y, z') \| \\
        & \leq \Lip(\pcs \circ (g - T)) \| z - z' \|,
    \end{align*}
    where in the second step we used $\pcs \circ T = T \circ \pcs$ to claim $\Tcs(y, z) = \Tcs(y, z')$.
    Since $\Lip(\varphi) \leq 1$ and $\Lip(g - T) \leq \varepsilon$, we continue from~\eqref{eq:h_y_varphi_contraction_bound} and, as announced, obtain that 
    \begin{align*}
         \mu \| \hyvarphi(z) - \hyvarphi(z') \| &\leq 2 \varepsilon \| z - z' \|.
    \end{align*}

    For the second part of the lemma, fix $z \in \Eu$.
    For any $y, y' \in \Ecs$, we proceed similarly to Eq.~\eqref{eq:h_y_varphi_contraction_bound}, and we obtain
    \begin{align} \label{eq:h_y_varphi_lipschitz_param}
        \mu \| h_y^\varphi(z) - h_{y'}^\varphi(z) \| 
        & \leq \big\| \TEu (h_y^\varphi(z) - h_{y'}^\varphi(z)) \big\| \notag \\ 
        & \leq \big\|  \varphi(\gcs(y, z)) -\varphi(\gcs(y', z)) \big\| + \big\| (\gu - \Tu)(y, z) - (\gu - \Tu)(y', z) \big\| \notag \\ 
        & \leq \Lip(\varphi) \; \| \gcs(y, z) - \gcs(y', z) \| + \Lip(\pu \circ (g - T)) \; \| y - y' \|.
    \end{align}
    For the first term, splitting again $g = (g - T) + T$ and using that $\Lip(\Tcs) \leq \lambda$, we have 
    \begin{align*}
        \| \gcs(y, z) - \gcs(y', z) \| 
        & \leq \| (\gcs - \Tcs) (y, z) - (\gcs - \Tcs)(y', z) \| + \| \Tcs(y, z) - \Tcs(y', z) \| \\ 
        & \leq \Lip(\pcs \circ (g - T)) \; \| y - y' \| + \lambda \| y - y' \| 
    \end{align*}
    Since $\Lip(g - T) \leq \varepsilon$, plugging this back in~\eqref{eq:h_y_varphi_lipschitz_param}, we obtain
    \begin{align*}
        \mu \| h_y^\varphi(z) - h_{y'}^\varphi(z) \| \leq
        (\lambda + 2 \varepsilon) \| y - y' \|,
    \end{align*}
    which shows that $y \mapsto \hyvarphi(z)$ is $(\lambda + 2 \varepsilon) / \mu$ Lipschitz.
\end{proof}

\subsection{Proof of Lemma~\ref{lemma:gamma-lipschitz}} \label{subsec:gamma-lipschitz}
\begin{proof} 
    We first show that $\Gamma$ is well defined.
    Let $(g, T) \in \mathrm{PH}(\mu, \lambda, \varepsilon; \reals^d, \Ecs \oplus \Eu)$.
    From the first part of Lemma~\ref{lemma:h_y_phi_lipschitz}, the map $z \mapsto h_y^\varphi(z)$ is a contraction, so $\Fix(h_y^\varphi)$ exists and is unique.
    Then, to check that $\Gamma \varphi \in \calF_1$, first observe that since $g(0) = 0$ and $\varphi(0) = 0$, then $h_y^\varphi(0) = 0$, so it holds indeed that $(\Gamma \varphi)(0) = 0$.
    From the uniform contraction principle (Lemma~\ref{lemma:uniform-contraction}), we obtain that $\Gamma \varphi$ is $(\lambda + 2 \varepsilon) / (\mu - 2 \varepsilon)$-Lipschitz.
    By~\aref{assumption:small-lipschitz}, this constant is less than $1$, which confirms that $\Gamma \varphi \in \calF_1$.

    We now show that $\Gamma$ is a contraction.
    Let $\varphi_1, \varphi_2 \in \calF_1$ and $y \in \Ecs$ be arbitrary.
    For $i \in \{1, 2\}$, define $z_i \coloneq (\Gamma \varphi_i)(y)$.
    Then, since $(\Gamma \varphi_i)(y) = \Fix(h_y^\varphi)$, it follows that $z_i = h_y^{\varphi_i}(z_i)$, so we have that 
    \begin{align*}
        \| (\Gamma \varphi_1)(y) - (\Gamma \varphi_2)(y) \| = 
        \| z_1 - z_2 \| = 
        \| h_y^{\varphi_1}(z_1) - h_y^{\varphi_2}(z_2) \|.
    \end{align*}
    Since $(T\vert_{\Eu}) h_y^{\varphi_i}(z_i) = \varphi_i(\gcs(y, z_i)) - (\gu - \Tu)(y, z_i)$ and using~\aref{assumption:invariance}, it follows that 
    \begin{align} \label{eq:Gamma-contraction-start}
        \mu \| z_1 - z_2 \| 
        & \leq \| \varphi_1(\gcs(y, z_1)) - \varphi_2(\gcs(y, z_2)) \| 
        + \| (\gu-\Tu)(y, z_1) - (\gu-\Tu)(y, z_2) \|  \notag \\ 
        & \leq  \| \varphi_1(\gcs(y, z_1)) - \varphi_2(\gcs(y, z_2)) \| + 
            \varepsilon \| (y, z_1) - (y, z_2) \|, 
    \end{align}
    where in the last step we used~\aref{assumption:small-lipschitz}.
    Rearranging, we obtain 
    \begin{align}
        (\mu - \varepsilon) \| z_1 - z_2 \| \leq \| \varphi_1(\gcs(y, z_1)) - \varphi_2(\gcs(y, z_2)) \|.
        \label{eq:Gamma-contraction-rearranged}
    \end{align}
    We continue by adding and subtracting $\varphi_1(\gcs(y, z_2))$, which yields
    \begin{multline}
        \| \varphi_1(\gcs(y, z_1)) - \varphi_2(\gcs(y, z_2)) \|
        \leq \| \varphi_1(\gcs(y, z_1)) - \varphi_1(\gcs(y, z_2)) \| \\
        + \| \varphi_1(\gcs(y, z_2)) - \varphi_2(\gcs(y, z_2)) \|.
        \label{eq:diff_phi_bound}
    \end{multline}
    We bound the first term of Eq.~\eqref{eq:diff_phi_bound} as
    \begin{align} \label{eq:same_phi_different_z_bound}
        \| \varphi_1(\gcs(y, z_1)) - \varphi_1(\gcs(y, z_2)) \|
        &\leq \Lip(\varphi_1) \, \| \gcs(y, z_1) - \gcs(g(y, z_2)) \| \notag \\
        &\leq \| (\gcs-\Tcs)(y, z_1) - (\gcs-\Tcs)(y, z_2) \| \notag \\
        &\phantom{\leq}\; + \| \Tcs(y, z_1) - \Tcs(y, z_2) \| \notag \\
        &\leq \Lip(g-T) \, \| (y, z_1) - (y, z_2) \| \notag \\ 
        &\phantom{\leq}\; + \| T(\pcs(y, z_1)) - T(\pcs(y, z_2)) \| \notag \\
        &\leq \varepsilon \, \| z_1 - z_2 \| + 0.
    \end{align}
    For the second term in Eq.~\eqref{eq:diff_phi_bound}, splitting $g = (g-T)+T$ as usual, 
    and using \aref{assumption:fixed-origin}, \aref{assumption:pseudo-hyperbolic} and~\aref{assumption:small-lipschitz}, we have 
    \begin{align*} 
        \| \varphi_1(\gcs(y, z_2)) - \varphi_2(\gcs(y, z_2)) \| 
        &\leq \| \varphi_1 - \varphi_2 \| \; \| \gcs(y, z_2) \| \notag \\ 
        &\leq \| \varphi_1 - \varphi_2 \| \; \Big( \| (\gcs - \Tcs)(y, z_2) - (\gcs-\Tcs)(0, 0) \| \notag \\ 
        &\phantom{\leq} \qquad\qquad\quad\quad  + \| \Tcs(y, z_2) \| \Big) \notag \\
        &\leq \| \varphi_1 - \varphi_2 \| \; \Big( \varepsilon \| (y, z_2) \| + \lambda \| y \| \Big).
    \end{align*}
    Since $\Gamma \varphi_2 \in \calF_1$, it follows that $(\Gamma \varphi_2)(0) = 0$ and $\Lip(\Gamma \varphi_2) \leq 1$, so 
    \begin{align*}
        \| z_2 \| = \| (\Gamma \varphi_2)(y) - (\Gamma \varphi_2)(0) \| \leq \| y \|
    \end{align*}
    Therefore, continuing from above and using the definition of the norm from Eq.~\eqref{eq:Rd-norm}, 
    we obtain 
    \begin{align} \label{eq:different_phi_same_z_bound}
        \| \varphi_1(\pcs(g(y, z_2))) - \varphi_2(\pcs(g(y, z_2))) \| \leq 
        \big( \varepsilon + \lambda \big) \; \| \varphi_1 - \varphi_2 \| \; \| y \|.
    \end{align}
    Plugging the bounds from Eq.~\eqref{eq:same_phi_different_z_bound} and~\eqref{eq:different_phi_same_z_bound} back in Eq.~\eqref{eq:diff_phi_bound}, we obtain
    \begin{align*}
         \| \varphi_1(\pcs(g(y, z_1))) - \varphi_2(\pcs(g(y, z_2))) \|
         &\leq \varepsilon \| z_1 - z_2 \| + \big( \varepsilon + \lambda \big) \; \| \varphi_1 - \varphi_2 \| \; \| y \|. 
    \end{align*}
    Finally, substituting this into Eq.~\eqref{eq:Gamma-contraction-rearranged}, using $z_i = (\Gamma \varphi_i)(y)$ yields
    \begin{align*}
        (\mu - 2 \varepsilon) \; \|  (\Gamma \varphi_1)(y) -  (\Gamma \varphi_2)(y) \| \leq 
        (\varepsilon + \lambda) \; \| \varphi_1 - \varphi_2 \| \; \| y \|
    \end{align*}
    Rearranging and taking the supremum over all $y \neq 0$, we obtain
    \begin{align*}
        \| \Gamma \varphi_1 - \Gamma \varphi_2 \| = 
        \sup_{y \neq 0} \frac{ \| (\Gamma \varphi_1)(y) - (\Gamma \varphi_2)(y) \| }{\| y \|}
        \leq \frac{\lambda + \varepsilon}{\mu - 2 \varepsilon} \; \| \varphi_1 - \varphi_2 \|.
    \end{align*}
    Observe that $(\lambda + \varepsilon) / (\mu - 2 \varepsilon) < 1$ from the pseudo-hyperbolicity assumptions.
\end{proof}

\subsection{Proof of Lemma~\ref{lemma:potential-growth}} \label{sec:proof-lemma-potential-growth}
\begin{proof} 
    Fix $x = (y, z) \in \reals^d$ and consider the modified point $x' = (y, (\Gamma \varphi)(y))$.
    Since $x' \in \graph(\Gamma \varphi)$, it follows from the invariance property in 
    Eq.~\eqref{eq:graph-invariance-gamma} that $g(x') \in \graph(\varphi)$, or, equivalently, $\pu(g(x')) = \varphi(\pcs(g(x')))$.
    We use this identity to decompose $V_\varphi(g(x))$ as follows 
    \begin{align*}
        V_\varphi(g(x))
        & = \| \pu(g(x)) - \varphi(\pcs(g(x))) \| \\
        & = \|  \pu(g(x)) - \pu(g(x')) + \varphi(\pcs(g(x'))) - \varphi(\pcs(g(x))) \| \\
        & \geq \| \pu(g(x)) - \pu(g(x')) \| - \| \varphi(\pcs(g(x))) - \varphi(\pcs(g(x'))) \|.
    \end{align*}
    We bound the two terms separately.
    For the first one, using $g = (g - T) + T$, we obtain 
    \begin{align*}
        \| \pu(g(x)) - \pu(g(x')) \|
        & \geq \| \pu(T(x)) - \pu(T(x')) \| - \| \pu((g-T)(x)) - \pu((g-T)(x')) \| \\ 
        & \geq \mu \| x - x' \| - \varepsilon \| x - x' \| 
        = (\mu - \varepsilon) \| x - x' \|.
    \end{align*}
    For the second term, using again $g = (g - T) + T$, we have 
    \begin{align*}
        \| \varphi(\pcs(g(x))) - \varphi(\pcs(g(x'))) \| 
        & \leq \Lip(\varphi) \Big( \| \pcs((g-T)(x)) - \pcs((g-T)(x')) \| \\ 
        & \qquad\qquad + \| \pcs(T(x)) - \pcs(T(x'))\| \Big) \\ 
        & \leq \varepsilon \| x - x' \|,
    \end{align*}
    where in the last step we used $ \pcs \circ T = T \circ \pcs$ and $\pcs(x) = \pcs(x')$ to claim that the second term is zero.
    Combining these two, we obtain 
    \begin{align*}
        V_\varphi(g(x)) 
        & \geq (\mu - 2 \varepsilon) \; \| x - x' \| \\ 
        & = (\mu - 2 \varepsilon) \; \| \pu(x) - (\Gamma \varphi)(\pcs(x)) \| \\ 
        & = (\mu - 2 \varepsilon) \; V_{\Gamma \varphi}(x),
    \end{align*}
    which is the claimed growth condition.
    To see that $\mu - 2 \varepsilon > 1$, use~\aref{assumption:small-lipschitz} and $\lambda \geq 1$ from Definition~\ref{def:pseudo-hyperbolicity}.
\end{proof}

\ifbool{BanachAppendix}{

\section{A Banach space of functions}

The following simple fact is used in Section~\ref{section:function-space}.
Recall that a Banach space is a complete normed vector space.

\begin{proposition} \label{prop:banach-space}
Let $\Ecs$ and $\Eu$ be Banach spaces, and define
\begin{align*}
    \calF = \{ \varphi \colon \Ecs \to \Eu \mid \varphi(0) = 0 \textrm{ and } \ \|\varphi\| < \infty \}, && \textrm{ with } &&
    \| \varphi \| = \sup_{y \neq 0} \frac{\|\varphi(y)\|}{\|y\|},
\end{align*}
where $\|\cdot\|$ also denotes the norms on $\Ecs$ and $\Eu$.
Then $(\calF, \|\cdot\|)$ is a Banach space.
\end{proposition}

\begin{proof}
It is straightforward to verify that $\|\cdot\|$ defines a norm on $\calF$.
We must show that $\calF$ is complete with that norm.
To this end, let $(\varphi_n)_{n \ge 1} \subset \calF$ be a Cauchy sequence with respect to $\|\cdot\|$.
For each fixed $y \in \Ecs$, we have
\begin{align*}
    \| \varphi_n(y) - \varphi_m(y) \| = \| (\varphi_n - \varphi_m)(y) \| \le \| \varphi_n - \varphi_m \| \, \| y \|,
\end{align*}
so $( \varphi_n(y) )_{n\ge 1}$ is a Cauchy sequence in $\Eu$.
By completeness of $\Eu$, the sequence $(\varphi_n(y))_{n \geq 1}$ converges: call the limit $\varphi(y)$.
Notice $\varphi(0) = \lim_{n \to \infty} \varphi_n(0) = 0$.
Since $\varphi(y) - \varphi_n(y) = \lim_{m \to \infty} (\varphi_m(y) - \varphi_n(y))$ for any $n \geq 0$, taking norms and dividing by $\| y \|$ for $y \neq 0$, we obtain 
\begin{align}
    \frac{\| \varphi(y) - \varphi_n(y) \|}{\| y \|} = 
    \lim_{m \to \infty} \frac{\| \varphi_m(y) - \varphi_n(y) \|}{\| y \|} 
    \leq \limsup_{m \geq n} \|\varphi_m - \varphi_n\|,
\end{align}
which tends to zero as $n \to \infty$ because $( \varphi_n )$ is Cauchy.
Taking the supremum over $y \neq 0$, we conclude that
$\|\varphi - \varphi_n\| \to 0$.
In particular, $\|\varphi\| = \lim_{n \to \infty} \| \varphi_n \| < \infty$, so $ \varphi \in \calF $.  
Hence every Cauchy sequence in $(\calF, \|\cdot\|)$ converges to a limit in $\calF$, and the space is complete.
\end{proof}

}{
}

\bibliographystyle{abbrvnat}
\bibliography{../bibtex/boumal}

\end{document}